\documentclass[twoside]{amsart} 

\usepackage{preamble} 

\usepackage{fancyhdr}
\title{Characterising slopes for hyperbolic knots\\ and Whitehead doubles} 
\author{Laura Wakelin}
\address{King's College London, Strand, London, WC2R 2LS}
\email{laura.1.wakelin@kcl.ac.uk}
\urladdr{https://sites.google.com/view/laurawakelin}

\begin{document}

\pagestyle{fancy}
\fancyhead{}
\fancyhead[CE]{\scriptsize\uppercase{Laura Wakelin}} 
\fancyhead[CO]{\scriptsize\uppercase{Characterising slopes for hyperbolic knots and Whitehead doubles}}

\fancyfoot[C]{\scriptsize\thepage}

\begin{abstract}
    A slope $p/q \in \mathbb{Q}$ is characterising for a knot $K \subset \mathbb{S}^3$ if the oriented homeomorphism type of the manifold $\mathbb{S}^3_K(p/q)$ obtained by Dehn surgery of slope $p/q$ on $K$ uniquely determines the knot $K$. We combine analysis of JSJ decompositions with techniques involving lengths of shortest geodesics to find explicit conditions for a slope to be characterising for $K$ in the case where $K$ is any hyperbolic knot or any satellite knot by a hyperbolic pattern. Assuming that the list of 2-cusped orientable hyperbolic 3-manifolds obtained using the computer programme SnapPy is complete up to a certain point, we use hyperbolic volume inequalities to generate a refinement for the special case of Whitehead doubles. We also construct pairs of multiclasped Whitehead doubles of double twist knots for which $1/q$ is a non-characterising slope. 
\end{abstract}

\maketitle

\section{Introduction} 
\label{section:introduction} 

Dehn surgery is a bridge between knot theory and 3-manifolds. 
It is well-known that every closed orientable 3-manifold can be obtained by surgery on a link in the 3-sphere $\mathbb{S}^3$ \cite{lickorish, wallace}. 
This is an existence result for which uniqueness fails: all such surgery descriptions can be related by some sequence of moves \cite{kirby}. 
However, if we restrict to the case of surgery on knots, then there are several different angles from which one can investigate the question of uniqueness. 
The cosmetic surgery conjecture postulates that there cannot exist an orientation-preserving homeomorphism between surgeries of different slopes on the same (non-trivial) knot unless there is a self-homeomorphism of the knot complement taking one slope to the other. 
The Dehn surgery characterisation problem, on the other hand, asks whether it is possible to have an orientation-preserving homeomorphism between surgeries of the same (non-trivial) slope on different knots. 

\begin{definition}
    A slope $p/q \in \mathbb{Q}$ is \emph{characterising} for a knot $K \subset \mathbb{S}^3$ if for any knot $K' \subset \mathbb{S}^3$, the existence of an orientation-preserving homeomorphism $\mathbb{S}^3_K(p/q) \cong \mathbb{S}^3_{K'}(p/q)$ implies that $K=K'$. 
\end{definition} 

Every slope is characterising for the unknot, $U$ \cite{kmos}, as well as for the trefoils, $\pm T = \pm 3_1$, and the figure eight knot, $S = 4_1$ \cite{os}. Recently, it was shown that at least all the non-integer slopes are characterising for the next twist knots, $\pm R = \pm 5_2$ \cite{bs}; however, little is known about characterising slopes for the stevedore knot, $6_1$. There are also examples of knots with non-characterising slopes, such as the hyperbolic knot $8_6$ (for which all integers are non-characterising) \cite{bm}. However, Lackenby proved that every knot $K$ has infinitely many characterising slopes in the interval $[-1,1]$; moreover, when $K$ is hyperbolic, there exists a constant $C(K)$ for which every slope $p/q$ with $|q|\geq C(K)$ is characterising for $K$ \cite{lackenby}. Sorya recently proved that $C(K)$ exists for any knot $K$ \cite{sorya}. 

We can write down an explicit formula for $C(K)$ when $K$ is a torus knot \cite{nz, mccoy} and Sorya showed that setting $C(K)=2$ works when $K$ is a composite knot \cite{sorya}. However, the proof of the existence of $C(K)$ in the general case is non-constructive, which motivates this primary purpose of this article: to find concrete values for $C(K)$ when $K$ is either a hyperbolic knot or a satellite knot by a hyperbolic pattern. 

Given any compact orientable hyperbolic 3-manifold $Y$ with toroidal boundary, let $\sys(Y)$ denote its \emph{systole} (the length of a shortest geodesic in $Y$). We first introduce the quantity 
$$\mathfrak{q}(Y) := \left\lceil \sqrt{6\sqrt{3} \bigg( 1.9793 \frac{2\pi}{\sys(Y)} + 28.78 \bigg)} \ \right\rceil$$ which will arise in our search for $C(K)$. Note that this exceeds $35$ when $\sys(Y)<0.14$. 

Our main result is that for any hyperbolic knot $K$, we have a formula to compute an explicit (though perhaps not optimal) value for $C(K)$. Let $\mathbb{S}^3_K$ denote the complement of $K$. 

\begin{theorem}
\label{theorem:new}
    Let $K$ be a hyperbolic knot. 

    Then every slope $p/q$ with $|q| \geq \max\{35, \mathfrak{q}(\mathbb{S}^3_K)\}$ is characterising for $K$. 
\end{theorem} 

As an example, consider the hyperbolic twist knots, such as the stevedore knot, $6_1$.  

\begin{corollary}
\label{corollary:new} 
    Let $K=T^\pm_t$ be a hyperbolic twist knot with $|t|$ full twists. 
    \begin{itemize}
        \item If $|t|\leq3$, then every slope $p/q$ with $|q|\geq 35$ is characterising for $T^\pm_t$. 
        \item If $|t|\geq4$, then every slope $p/q$ with $|q|\geq \mathfrak{q}(\mathbb{S}^3_{T^\pm_t})$ is characterising for $T^\pm_t$.  
    \end{itemize} 
\end{corollary} 

The same idea can be used to address any satellite knot by a hyperbolic pattern $P = Q \cup U$ (where $Q \subset \mathbb{S}^1 \times \mathbb{D}^2$ is a hyperbolic knot in the solid torus which we will use to build the satellite knot). The \emph{winding number} $w$ of the pattern $P$ is taken to mean that of $Q \subset \mathbb{S}^1\times\mathbb{D}^2$. 

\begin{theorem}
\label{theorem:main}
    Let $K=P(J)$ be a satellite of a companion $J$ by a hyperbolic pattern $P = Q \cup U$ with winding number $w$. 

    Then every slope $p/q$ with $\gcd(p,w)\neq1$ and $|q|\geq\max\{35, \mathfrak{q}(\mathbb{S}^3_P)\}$ is characterising for $K$.
\end{theorem}

\begin{figure}[htbp!]
    \centering 
    \resizebox{0.25\textwidth}{!}{\centering
\begin{tikzpicture} [squarednode/.style={rectangle, draw=black, fill=white, thick, minimum size=40pt}]
    \begin{knot}[end tolerance=1pt] 
    \flipcrossings{1,4} 
        \strand[white, double=black, thick, double distance=1pt] (-1,-1.25) 
        to [out=left, in=down] (-2.5,0) 
        to [out=up, in=left] (-1,3)
        to [out=right, in=up] (-0.25,2.5)
        to [out=down, in=up] (-0.25,1) 
        to [out=down, in=up] (-2,0)
        to [out=down, in=left] (-1,-0.75); 
        \strand[white, double=black, thick, double distance=1pt] (-1,-0.75) to [out=right, in=left] (1,-0.75); 
        \strand[white, double=black, thick, double distance=1pt] (1,-0.75) 
        to [out=right, in=down] (2,0)
        to [out=up, in=down] (0.25,1)
        to [out=up, in=down] (0.25,2.5) 
        to [out=up, in=left] (1,3) 
        to [out=right, in=up] (2.5,0) 
        to [out=down, in=right] (1,-1.25); 
        \strand[white, double=black, thick, double distance=1pt] (1,-1.25) to [out=left, in=right] (-1,-1.25); 
        \strand[white, double=black, thick, double distance=1pt] (-3.25,0) 
        to [out=down, in=down] (-1.25,0) 
        to [out=up, in=up] (-3.25,0); 
    \end{knot} 
    \node[squarednode, fill=white, opacity=1] (A) at (0,-1) {\Large $-t$}; 
    \node[squarednode, fill=white, opacity=1] (B) at (0,1.75) {\Large $-n$}; 
\end{tikzpicture} } 
    \caption{The $n$-clasped $t$-twisted Whitehead link, $W^n_t = V^n_t \cup U$. } 
    \label{figure:$W^n_t$} 
\end{figure}
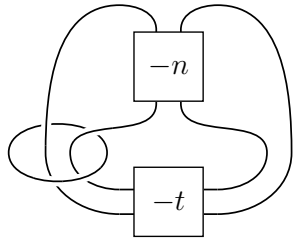 

In particular, the result applies to the special case of Whitehead doubles (which have $w=0$). 

\begin{corollary} 
\label{corollary:main} 
    Let $K=W^n_t(J)$ be an $n$-clasped $t$-twisted Whitehead double. 
    \begin{itemize}  
        \item If $|n|\leq4$, then every slope $p/q$ with $|p|\neq1$ and $|q|\geq35$ is characterising for $K$. 
        \item If $|n|\geq5$, then every slope $p/q$ with $|p|\neq1$ and $|q|\geq\mathfrak{q}(\mathbb{S}^3_{W^n_t})$ is characterising for $K$. 
    \end{itemize}
\end{corollary} 

Our strategy uses JSJ decompositions to reduce the problem to one about the pattern $P$, which can be addressed using lengths of minimal geodesics. Moreover, when the pattern is given by the Whitehead link, we can instead employ an alternative argument which relies on hyperbolic volume: crucially, the Whitehead link complement and its sister are known to be the smallest amongst all 2-cusped orientable hyperbolic 3-manifolds \cite{agol}. This leads to a better bound, but relies on the conjectural completeness of the SnapPy census \cite{SnapPy} of such manifolds up to a given point. That is to say, let $V_k$ denote the $k^\text{th}$ volume which appears and let $a_k$ denote the number of SnapPy census manifolds with volume at most $V_k$. We say that the SnapPy census is \emph{complete up to stage $k$} if it includes every 2-cusped orientable hyperbolic 3-manifold of volume at most $V_k$. 

\begin{theorem} 
\label{theorem:refined}
    Let $K=W^\pm_t(J)$ be a $\pm1$-clasped $t$-twisted Whitehead double for \mbox{$\pm t\in \{-1, 0, 1, 2\}$}. 
    
    Then every slope $p/q$ with $|p|\neq1$ and $|q|\geq q_{\min}$ is characterising for $K$, where $q_{\min}$ is some constant determined by the Whitehead link and independent of $J$. 

    Furthermore, if the SnapPy census of 2-cusped orientable hyperbolic 3-manifolds is complete up to stage $k$ of Table \ref{table:stages} (namely, up to the volume bound $V_k$, or up to the first $a_k$ items), then we can take $q_{\min}$ to be $q_k$. 
\end{theorem} 

\begin{table}[htbp!]
\centering
    \begin{tabular}{|c|c|c|c|}
        \hline 
        $k$ & $V_k$ & $a_k$ & $q_k$ \\ \hline \hline 
        $1$ & $4.0464$ & $4$ & $43$ \\ \hline 
        $2$ & $4.4057$ & $6$ & $32$ \\ \hline 
        $3$ & $4.5275$ & $8$ & $30$ \\ \hline 
        $4$ & $4.6842$ & $10$ & $28$ \\ \hline 
        $5$ & $4.7801$ & $16$ & $27$ \\ \hline 
        $6$ & $4.8913$ & $18$ & $26$ \\ \hline 
        $7$ & $5.0212$ & $24$ & $25$ \\ \hline 
        $8$ & $5.1749$ & $57$ & $24$ \\ \hline 
    \end{tabular} 
\caption{The values of $V_k$, $a_k$ and $q_k$ for each stage $k$.}  
\label{table:stages}
\end{table} 

In order to demonstrate the reason for imposing a condition on the numerator of the slope, for each $1/q$ we construct an infinite family of pairs of distinct multiclasped Whitehead doubles of double twist knots sharing a surgery of slope $1/q$. These are depicted in Figure \ref{figure:$W^n(T^m_q)$} (where the extra twisting compensates for the writhe of the companion knot). A value of $C(K)$ for such a knot $K$ must therefore involve the companion knot. 

\begin{theorem} 
\label{theorem:non-characterising}
    Let $K = W^n(T^m_q)$ and $K' = W^m(T^n_q)$ for any choices of $q,m,n\in\mathbb{Z}\setminus\{0\}$. 

    Then there exists an orientation-preserving homeomorphism $\mathbb{S}^3_K(1/q)\cong\mathbb{S}^3_{K'}(1/q)$. 

    In particular, when $m \neq n$, $1/q$ is a non-characterising slope for both $K$ and $K'$. 
\end{theorem} 

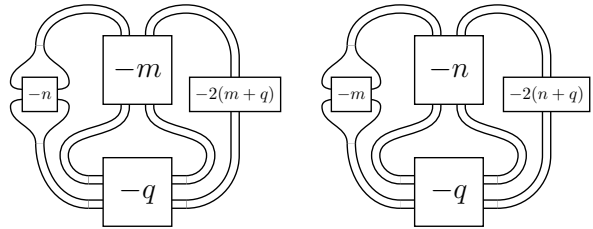
\begin{figure}[htbp!] 
    \centering
    \resizebox{0.6\textwidth}{!}{\centering
\begin{tikzpicture} [squarednode/.style={rectangle, draw=black, fill=white, thick, minimum size=40pt}]
    \begin{knot} 
        \strand[black, double=white, thick, double distance=4pt] (-1,-1) 
        to [out=left, in=down] (-2,0.25);  
        \strand[black, thick] (-2.085,0.25) 
        to [out=up, in=down] (-2.6,1)
        to [out=up, in=left] (-2.3,1.15); 
        \strand[black, thick] (-2.3,1.35) 
        to [out=left, in=down] (-2.6,1.5) 
        to [out=up, in=down] (-2.085,2.25); 
        \strand[black, thick] (-1.915,0.25) 
        to [out=up, in=down] (-1.4,1)
        to [out=up, in=right] (-1.7,1.15); 
        \strand[black, thick] (-1.7,1.35) 
        to [out=right, in=down] (-1.4,1.5) 
        to [out=up, in=down] (-1.915,2.25);         
        \strand[black, double=white, thick, double distance=4pt] (-2,2.25) 
        to [out=up, in=left] (-1,3)
        to [out=right, in=up] (-0.25,2.5)
        to [out=down, in=up] (-0.25,1) 
        to [out=down, in=up] (-1.5,0)
        to [out=down, in=left] (-1,-0.5); 
        \strand[black, double=white, thick, double distance=4pt] (-1,-0.5) 
        to [out=right, in=left] (1,-0.5); 
        \strand[black, double=white, thick, double distance=4pt] (1,-0.5) 
        to [out=right, in=down] (1.5,0)
        to [out=up, in=down] (0.25,1)
        to [out=up, in=down] (0.25,2.5) 
        to [out=up, in=left] (1,3) 
        to [out=right, in=up] (2,2) 
        to [out=down, in=up] (2,1.6)
        to [out=down, in=up] (2,0.9) 
        to [out=down, in=up] (2,0) 
        to [out=down, in=right] (1,-1); 
        \strand[black, double=white, thick, double distance=4pt] (1,-1) 
        to [out=left, in=right] (-1,-1); 
    \end{knot} 
    \node[squarednode, fill=white, opacity=1] at (0,-0.75) {\Large $-q$}; 
    \node[squarednode, fill=white, opacity=1] at (0,1.75) {\Large $-m$}; 
    \node[rectangle, draw=black, fill=white, fill opacity=1, thick, minimum size=20pt] at (-2,1.25) {$-n$}; 
    \node[rectangle, draw=black, fill=white, fill opacity=1, thick, minimum size=20pt] at (2,1.25) {$-2(m+q)$}; 
\end{tikzpicture} 
\qquad 
\begin{tikzpicture} [squarednode/.style={rectangle, draw=black, fill=white, thick, minimum size=40pt}]
    \begin{knot} 
        \strand[black, double=white, thick, double distance=4pt] (-1,-1) 
        to [out=left, in=down] (-2,0.25);  
        \strand[black, thick] (-2.085,0.25) 
        to [out=up, in=down] (-2.6,1)
        to [out=up, in=left] (-2.3,1.15); 
        \strand[black, thick] (-2.3,1.35) 
        to [out=left, in=down] (-2.6,1.5) 
        to [out=up, in=down] (-2.085,2.25); 
        \strand[black, thick] (-1.915,0.25) 
        to [out=up, in=down] (-1.4,1)
        to [out=up, in=right] (-1.7,1.15); 
        \strand[black, thick] (-1.7,1.35) 
        to [out=right, in=down] (-1.4,1.5) 
        to [out=up, in=down] (-1.915,2.25);         
        \strand[black, double=white, thick, double distance=4pt] (-2,2.25) 
        to [out=up, in=left] (-1,3)
        to [out=right, in=up] (-0.25,2.5)
        to [out=down, in=up] (-0.25,1) 
        to [out=down, in=up] (-1.5,0)
        to [out=down, in=left] (-1,-0.5); 
        \strand[black, double=white, thick, double distance=4pt] (-1,-0.5) 
        to [out=right, in=left] (1,-0.5); 
        \strand[black, double=white, thick, double distance=4pt] (1,-0.5) 
        to [out=right, in=down] (1.5,0)
        to [out=up, in=down] (0.25,1)
        to [out=up, in=down] (0.25,2.5) 
        to [out=up, in=left] (1,3) 
        to [out=right, in=up] (2,2) 
        to [out=down, in=up] (2,1.6)
        to [out=down, in=up] (2,0.9) 
        to [out=down, in=up] (2,0) 
        to [out=down, in=right] (1,-1); 
        \strand[black, double=white, thick, double distance=4pt] (1,-1) 
        to [out=left, in=right] (-1,-1); 
    \end{knot} 
    \node[squarednode, fill=white, opacity=1] at (0,-0.75) {\Large $-q$}; 
    \node[squarednode, fill=white, opacity=1] at (0,1.75) {\Large $-n$}; 
    \node[rectangle, draw=black, fill=white, fill opacity=1, thick, minimum size=20pt] at (-2,1.25) {$-m$}; 
    \node[rectangle, draw=black, fill=white, fill opacity=1, thick, minimum size=20pt] at (2,1.25) {$-2(n+q)$}; 
\end{tikzpicture} }
    \caption{A pair of knots, $K=W^n(T^m_q)$ and $K'=W^m(T^n_q)$, sharing a $1/q$-surgery. } 
    \label{figure:$W^n(T^m_q)$}
\end{figure}

\subsection{Outline}

The structure of this paper is as follows. In Section \ref{section:preliminaries}, we review some key concepts and relevant results. In Section \ref{section:jsj}, we use a JSJ decomposition argument to show that any knot sharing the same surgery as our fixed knot must be of a very similar type. In Section \ref{section:geodesics}, we employ a minimal geodesic argument to show that in fact the only possibility is for both knots to be exactly the same, thus completing the proofs of Theorems \ref{theorem:new} and \ref{theorem:main}. In Section \ref{section:volume}, we discuss hyperbolic volume and describe our alternative strategy for Whitehead doubles, which proves Theorem \ref{theorem:refined}. In Section \ref{section:non-characterising}, we construct explicit examples of pairs of multiclasped Whitehead doubles sharing a surgery of slope $1/q$ as described in Theorem \ref{theorem:non-characterising}.

\subsection{Acknowledgements} 

The author would like to thank Steven Sivek for his invaluable guidance, as well as Marc Lackenby and Patricia Sorya for many helpful conversations. 
She would additionally like to thank David Gabai and Duncan McCoy for their suggested improvements which inspired some of the new results in the final version of this article. 
This work was supported by the Engineering and Physical Sciences Research Council (EPSRC) Centre for Doctoral Training in Geometry and Number Theory at the Interface (The London School of Geometry and Number Theory), University College London [EP/S021590/1]. 
The author would also like to thank her host institution, Imperial College London, and the Max-Planck-Institut für Mathematik (MPIM) for their hospitality.

\section{Preliminaries} 
\label{section:preliminaries} 

We begin by outlining some of the basic ideas that we will draw on in this article. 
In particular, we review the different geometries of 3-manifolds and how these manifest for knot complements before and after Dehn filling. 
Throughout, we use the symbol $\cong$ to denote an orientation-preserving homeomorphism between 3-manifolds and $=$ to denote an ambient isotopy between knots or links. The notation $Q \cup U = Q' \cup U'$ corresponds to an isotopy of links which sends $Q$ to $Q'$ and $U$ to $U'$. Our 3-manifolds will always be connected and orientable and may have empty or toroidal boundary.

\subsection{Satellite knots} 

Recall that a \emph{knot} $K$ in a 3-manifold $M$ is the image of an embedding $\mathbb{S}^1 \hookrightarrow M$ considered up to ambient isotopy (whilst a \emph{link} $L$ is allowed to have multiple components). We denote mirror pairs of knots by $\pm K$. Drilling out a tubular neighbourhood $\nu(K)\cong\mathbb{S}^1\times\mathbb{D}^2$ of $K$ from $M$ produces the knot complement $M_K:=M \setminus\nu(K)$ (and similarly for links). In this discussion, we usually take our ambient manifold $M$ to be either $\mathbb{S}^3$ or $\mathbb{S}^1\times\mathbb{D}^2$, occasionally generalising to the complement $\mathbb{S}^3_{U^m}$ of an unlink with $m\geq0$ components -- the complement of the empty unlink is $\mathbb{S}^3$ itself. When $K$ is nullhomologous, we denote its \emph{meridian} by $\mu$ and (Seifert) \emph{longitude} by $\lambda$. This will be our canonical choice of generators for $H_1(\partial \nu(K))$; for links, we consider each component in isolation. We say that a simple closed curve $\gamma = p\mu+q\lambda$ on $\partial \nu(K) \cong \mathbb{T}^2$ (considered up to isotopy) has \emph{slope} $p/q \in \mathbb{Q} \cup \{\infty\}$, where $p,q\in\mathbb{Z}$ are a pair of coprime integers and $1/0:=\infty$. We may also refer to $\gamma$ itself as a slope. 

Knots in $\mathbb{S}^3$ beyond the unknot $U=U^1$ can be broadly classified into three main groups. A \emph{torus knot} $T_{a,b}$ is one which can be drawn on the surface of an unknotted torus, in the form of a simple closed curve winding $a$ times around the meridian and $b$ times around the longitude for some pair of coprime integers $a,b\in\mathbb{Z}$ with $|a|>1$ and $|b|>1$; relaxing the coprimeness condition extends this definition to also include \emph{torus links}. A \emph{hyperbolic knot} $K$ by definition is precisely one whose complement $\mathbb{S}^3_K$ admits a complete finite-volume hyperbolic metric; similarly, we can also define \emph{hyperbolic links}. Together, torus knots and hyperbolic knots form a class of knots called \emph{simple knots}. All other knots are called \emph{satellite knots} and can be constructed iteratively via the following process. 

\begin{definition}
    The \emph{satellite knot} $K=P(J)$ with \emph{pattern} $P = Q \cup U$ (for a knot $Q \subset \mathbb{S}^3_U\cong\mathbb{S}^1\times\mathbb{D}^2$ which is neither contained in any $\mathbb{B}^3\subset\mathbb{S}^3_U$ nor isotopic to the core of $\mathbb{S}^3_U$) and non-trivial \emph{companion knot} $J\subset \mathbb{S}^3$ is the image of $Q$ in the oriented homeomorphism $f: \mathbb{S}^3_U \to \nu(J)$ which sends the preferred longitude $\mu_U$ of the unknotted solid torus $\mathbb{S}^3_U \cong \mathbb{S}^1 \times \mathbb{D}^2$ to the longitude $\lambda_J$ of $\nu(J)$. 
\end{definition} 

Our work investigates both hyperbolic knots and satellite knots which can be expressed in the form $K=P(J)$ for some pattern $P = Q \cup U$ whose complement $\mathbb{S}^3_P$ is hyperbolic. 

As a special case, we focus on the \emph{$n$-clasped $t$-twisted Whitehead doubles}, $K=W^n_t(J)$, which can be constructed from any companion knot $J$ by applying the \emph{$n$-clasped $t$-twisted Whitehead link} pattern, $W^n_t = V^n_t \cup U$, for some integers $n\neq 0$ and $t\in\mathbb{Z}$ (as depicted in Figure \ref{figure:$W^n_t$}). One key property of this pattern is that it has linking number zero. Note that the component $V^n_t \subset \mathbb{S}^3_U$ is in fact isotopic to the \emph{$(n,t)$-double twist knot}, $T^n_t \subset \mathbb{S}^3$. Although $n$ and $t$ are interchangeable for $T^n_t \subset \mathbb{S}^3$, they are uniquely determined for $V^n_t \subset \mathbb{S}^3_U$. When $t=0$ (which we take unless otherwise stated), $T^n_0$ is unknotted; moreover, the unknotted components of $W^n = V^n \cup U$ can then be exchanged by an isotopy, which will be an important step in our construction of the non-characterising slopes in Theorem \ref{theorem:non-characterising}. 

Different values of $n$ give non-homeomorphic link complements, which can be distinguished for each $|n|$ by their hyperbolic volumes ($t$-twisting has no effect). When $n=\pm1$ (for which we often use the shorthand $\pm$), the complement of $W^\pm_t$ has the joint smallest volume amongst all 2-cusped orientable hyperbolic 3-manifolds \cite{agol}: this is essential for our proof of Theorem \ref{theorem:refined}. Furthermore, the component $V^\pm_t\subset\mathbb{S}^3_U$ is isotopic to a \emph{$t$-twist knot}, $T^\pm_t \subset \mathbb{S}^3$. The values of $t$ in the statement of the theorem come from those $T^\pm_t$, shown in Figure \ref{figure:twist}, for which all non-integral slopes are known to be characterising. 

There is another important family of satellite knots which will arise throughout our discussion. The $(r,s)$-\emph{cable} $K=C_{r,s}(J)$ of a knot $J$ is the satellite knot with companion knot $J$ and pattern \mbox{$C_{r,s} = T_{r,s} \cup U$}, where the torus knot $T_{r,s}\subset\mathbb{S}^3_U$ can be isotoped to lie on $\partial \mathbb{S}^3_U \cong\mathbb{T}^2$ in the standard way. Note that our convention is to take the winding number of the torus knot $T_{r,s}$ inside $\mathbb{S}^3_U$ to be $|s|$ rather than $|r|$. We include patterns with $|r|=1$ (despite $T_{1,s}$ being unknotted), but we always require $|s|>1$ (otherwise the pattern is trivial). 

Observe that there are often multiple ways to express a satellite knot in the form $P(J)$: for instance, a composite knot $J_1 \# J_2$ can be written either as a satellite $P_2(J_1)$ of $J_1$ by a pattern $P_2$ with $P_2(U)=J_2$, or as a satellite $P_1(J_2)$ of $J_2$ by a pattern $P_1$ with $P_1(U)=J_1$. Meanwhile, the satellite knots $C_{r,s}(W(J))$ can alternatively be thought of as $(C_{r,s}(W))(J)$.

\subsection{Geometry of 3-manifolds} 

In three dimensions, topology and geometry are closely entwined. There are two important classes of 3-manifolds which form the building blocks for this relationship. 

Firstly, we define \emph{Seifert fibred spaces}, which admit an $\mathbb{S}^1$-fibration over an orbifold base. 

\begin{definition} 
    A \emph{Seifert fibred space} $\widetilde{\Sigma}(\alpha_1/\beta_1, \ldots, \alpha_n/\beta_n)$ is an $\mathbb{S}^1$-bundle over an orbifold $\Sigma(\beta_1, \ldots, \beta_n)$ for which each $\mathbb{S}^1$-fibre admits a tubular neighbourhood that is homeomorphic to the $\mathbb{S}^1$-fibred mapping torus of a $2\pi t$-twist automorphism of $\mathbb{D}^2$ for some $t\in\mathbb{Q}$. The central $\mathbb{S}^1$-fibre is called \emph{regular} when $t \in \mathbb{Z}$ and \emph{exceptional} of order $|\beta_i|>1$ when $t=\alpha_i/\beta_i \not\in\mathbb{Z}$ for each pair of coprime integers $\alpha_i, \beta_i \in \mathbb{Z}$. 
\end{definition} 

For example, the complement $\mathbb{S}^3_{T_{a,b}}$ of the torus knot $T_{a,b}$ is a Seifert fibred space over the disc orbifold $\mathbb{D}^2(a,b)$, which has two exceptional fibres of orders $|a|$ and $|b|$; the Seifert invariants $a'$ and $b'$ satisfy the relationship $|ab'+ba'|=1$. Similarly, the complement $\mathbb{S}^3_{C_{r,s}}$ of the $(r,s)$-cable pattern $C_{r,s}$ is a Seifert fibred space over the annulus orbifold $\mathbb{A}^2(s)$, which has one exceptional fibre of order $|s|$ and Seifert invariant $r$ representing the behaviour of the fibre in the other direction -- this manifold is called a \emph{cable space}. As a final key example, a Seifert fibred space over a planar surface $\Sigma = \mathbb{S}^2 \setminus \sqcup_{i=1}^{m\geq3} \mathbb{D}^2$ with no exceptional fibres is called a \emph{composing space}. 

Secondly, we define \emph{hyperbolic 3-manifolds}, which we will later revisit in more detail. 

\begin{definition}
    A \emph{hyperbolic 3-manifold} $M \cong \mathbb{H}^3 / \Gamma$ is the quotient of hyperbolic 3-space $\mathbb{H}^3$ by a freely-acting discrete subgroup $\Gamma$ of the orientation-preserving isometry group $\mathrm{Isom}^+(\mathbb{H}^3)\cong \mathrm{PSL}(2,\mathbb{C})$. 
\end{definition} 

For instance, the figure eight knot complement can be built from gluing two ideal tetrahedra in $\mathbb{H}^3$, whilst the Whitehead link complement is constructed from a single ideal octahedron in $\mathbb{H}^3$ \cite{thurston}. 

\begin{theorem}[Hyperbolisation theorem {\cite[Theorem~2.3]{thurston-hyperbolisation}}]
\label{theorem:hyperbolisation}
    Let $M$ be a compact orientable 3-manifold with toroidal boundary. Then $M$ admits a hyperbolic structure if and only if it is irreducible and atoroidal with infinite $\pi_1(M)$. 
\end{theorem}  

This hyperbolic structure has finite volume provided $M$ is neither $\mathbb{S}^1 \times \mathbb{D}^2$ nor $\mathbb{T}^2 \times \mathbb{I}$. By Mostow rigidity \cite{mostow}, the geometry of a finite-volume hyperbolic 3-manifold is an invariant of its homeomorphism type, and thus gives us a tool for identification. Note that a finite-volume hyperbolic 3-manifold cannot simultaneously be a Seifert fibred space. However, a compact orientable 3-manifold $M$ can always be built from a combination of these types of pieces by first invoking the prime decomposition theorem \cite{kneser, milnor} and then applying the following result to each irreducible constituent.

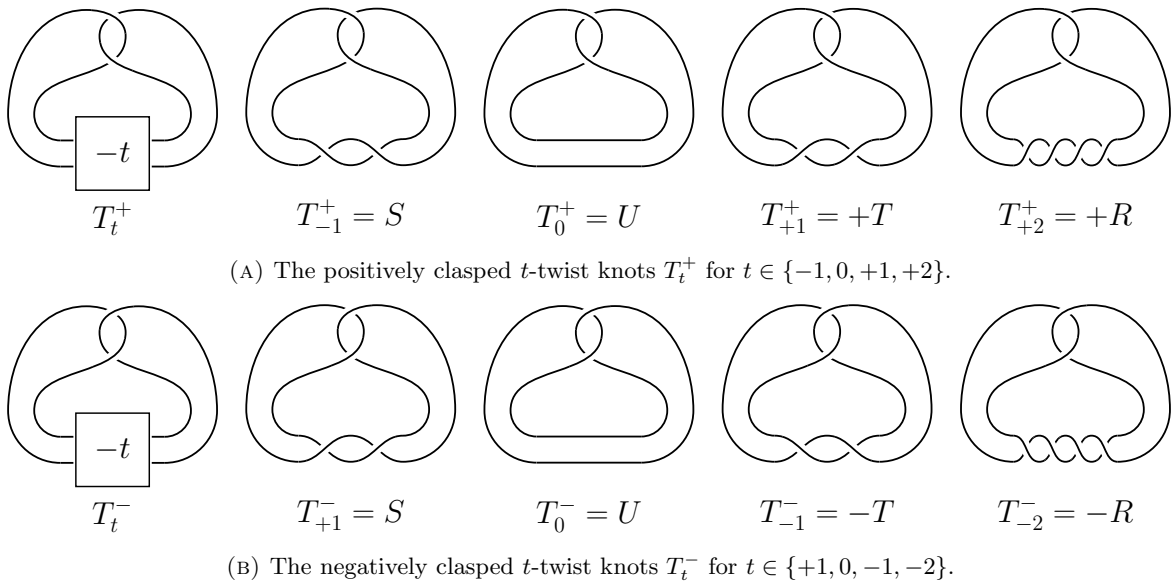
\begin{figure}[htbp!]
    \centering
    \begin{subfigure}{0.9\textwidth}
        \resizebox{\textwidth}{!}{\centering
\begin{tikzpicture} [squarednode/.style={rectangle, draw=black, fill=white, thick, minimum size=40pt}]
    \begin{knot} 
        \strand[white, double=black, thick, double distance=1pt] (-1,-1) 
        to [out=left, in=down] (-2,0)
        to [out=up, in=left] (-0.5,2)
        to [out=right, in=up] (0.25,1.5)
        to [out=down, in=up] (-1.5,0)
        to [out=down, in=left] (-1,-0.5); 
        \strand[white, double=black, thick, double distance=1pt] (-1,-0.5) to [out=right, in=left] (1,-0.5); 
        \strand[white, double=black, thick, double distance=1pt] (1,-0.5) 
        to [out=right, in=down] (1.5,0)
        to [out=up, in=down] (-0.25,1.5)
        to [out=up, in=left] (0.5,2)
        to [out=right, in=up] (2,0) 
        to [out=down, in=right] (1,-1); 
        \strand[white, double=black, thick, double distance=1pt] (1,-1) to [out=left, in=right] (-1,-1);
        \flipcrossings{2} 
        \node[] at (0,-2) {\Large $T^+_{t}$};
    \end{knot} 
    \node[squarednode] (A) at (0,-0.75) {\Large $-t$}; 
\end{tikzpicture} 
\quad 
\begin{tikzpicture} []
    \begin{knot} [] 
    \flipcrossings{2,4} 
        \strand[white, double=black, thick, double distance=1pt] (-1,-1) 
        to [out=left, in=down] (-2,0)
        to [out=up, in=left] (-0.5,2)
        to [out=right, in=up] (0.25,1.5)
        to [out=down, in=up] (-1.5,0)
        to [out=down, in=left] (-1,-0.5); 
        \strand[white, double=black, thick, double distance=1pt] (-1,-0.5) 
        to [out=right, in=left] (0,-1)
        to [out=right, in=left] (1,-0.5); 
        \strand[white, double=black, thick, double distance=1pt] (1,-0.5) 
        to [out=right, in=down] (1.5,0)
        to [out=up, in=down] (-0.25,1.5)
        to [out=up, in=left] (0.5,2)
        to [out=right, in=up] (2,0) 
        to [out=down, in=right] (1,-1); 
        \strand[white, double=black, thick, double distance=1pt] (1,-1) 
        to [out=left, in=right] (0,-0.5)
        to [out=left, in=right] (-1,-1); 
        \node[] at (0,-2) {\Large $T^+_{-1}=S$}; 
    \end{knot} 
\end{tikzpicture} 
\quad 
\begin{tikzpicture} []
    \begin{knot} [] 
    \flipcrossings{2} 
        \strand[white, double=black, thick, double distance=1pt] (-1,-1) 
        to [out=left, in=down] (-2,0)
        to [out=up, in=left] (-0.5,2)
        to [out=right, in=up] (0.25,1.5)
        to [out=down, in=up] (-1.5,0)
        to [out=down, in=left] (-1,-0.5); 
        \strand[white, double=black, thick, double distance=1pt] (-1,-0.5) 
        to [out=right, in=left] (1,-0.5); 
        \strand[white, double=black, thick, double distance=1pt] (1,-0.5) 
        to [out=right, in=down] (1.5,0)
        to [out=up, in=down] (-0.25,1.5)
        to [out=up, in=left] (0.5,2)
        to [out=right, in=up] (2,0) 
        to [out=down, in=right] (1,-1); 
        \strand[white, double=black, thick, double distance=1pt] (1,-1) 
        to [out=left, in=right] (-1,-1); 
        \node[] at (0,-2) {\Large $T^+_0=U$}; 
    \end{knot} 
\end{tikzpicture} 
\quad 
\begin{tikzpicture} [] 
    \begin{knot} [] 
    \flipcrossings{2,3} 
        \strand[white, double=black, thick, double distance=1pt] (-1,-1) 
        to [out=left, in=down] (-2,0)
        to [out=up, in=left] (-0.5,2)
        to [out=right, in=up] (0.25,1.5)
        to [out=down, in=up] (-1.5,0)
        to [out=down, in=left] (-1,-0.5); 
        \strand[white, double=black, thick, double distance=1pt] (-1,-0.5) 
        to [out=right, in=left] (0,-1)
        to [out=right, in=left] (1,-0.5); 
        \strand[white, double=black, thick, double distance=1pt] (1,-0.5) 
        to [out=right, in=down] (1.5,0)
        to [out=up, in=down] (-0.25,1.5)
        to [out=up, in=left] (0.5,2)
        to [out=right, in=up] (2,0) 
        to [out=down, in=right] (1,-1); 
        \strand[white, double=black, thick, double distance=1pt] (1,-1) 
        to [out=left, in=right] (0,-0.5)
        to [out=left, in=right] (-1,-1); 
        \node[] at (0,-2) {\Large $T^+_{+1}=+T$};
    \end{knot} 
\end{tikzpicture}
\quad 
\begin{tikzpicture} [] 
    \begin{knot} [] 
    \flipcrossings{2,3,5} 
        \strand[white, double=black, thick, double distance=1pt] (-1,-1) 
        to [out=left, in=down] (-2,0)
        to [out=up, in=left] (-0.5,2)
        to [out=right, in=up] (0.25,1.5)
        to [out=down, in=up] (-1.5,0)
        to [out=down, in=left] (-1,-0.5); 
        \strand[white, double=black, thick, double distance=1pt] (-1,-0.5) 
        to [out=right, in=left] (-0.5,-1)
        to [out=right, in=left] (0,-0.5)
        to [out=right, in=left] (0.5,-1)
        to [out=right, in=left] (1,-0.5); 
        \strand[white, double=black, thick, double distance=1pt] (1,-0.5) 
        to [out=right, in=down] (1.5,0)
        to [out=up, in=down] (-0.25,1.5)
        to [out=up, in=left] (0.5,2)
        to [out=right, in=up] (2,0) 
        to [out=down, in=right] (1,-1); 
        \strand[white, double=black, thick, double distance=1pt] (1,-1) 
        to [out=left, in=right] (0.5,-0.5)
        to [out=left, in=right] (0,-1)
        to [out=left, in=right] (-0.5,-0.5)
        to [out=left, in=right] (-1,-1); 
        \node[] at (0,-2) {\Large $T^+_{+2}=+R$};
    \end{knot} 
\end{tikzpicture} }
        \caption{The positively clasped $t$-twist knots $T^+_t$ for $t\in\{-1,0,+1,+2\}$. } 
        \label{subfigure:twist+} 
    \end{subfigure} 
    \begin{subfigure}{0.9\textwidth} 
        \resizebox{\textwidth}{!}{\centering
\begin{tikzpicture} [squarednode/.style={rectangle, draw=black, fill=white, thick, minimum size=40pt}]
    \begin{knot} 
        \strand[white, double=black, thick, double distance=1pt] (-1,-1) 
        to [out=left, in=down] (-2,0)
        to [out=up, in=left] (-0.5,2)
        to [out=right, in=up] (0.25,1.5)
        to [out=down, in=up] (-1.5,0)
        to [out=down, in=left] (-1,-0.5); 
        \strand[white, double=black, thick, double distance=1pt] (-1,-0.5) to [out=right, in=left] (1,-0.5); 
        \strand[white, double=black, thick, double distance=1pt] (1,-0.5) 
        to [out=right, in=down] (1.5,0)
        to [out=up, in=down] (-0.25,1.5)
        to [out=up, in=left] (0.5,2)
        to [out=right, in=up] (2,0) 
        to [out=down, in=right] (1,-1); 
        \strand[white, double=black, thick, double distance=1pt] (1,-1) to [out=left, in=right] (-1,-1); 
        \flipcrossings{1} 
        \node[] at (0,-2) {\Large $T^-_{t}$};
    \end{knot} 
    \node[squarednode] (A) at (0,-0.75) {\Large $-t$}; 
\end{tikzpicture} 
\quad 
\begin{tikzpicture} []
    \begin{knot} [] 
    \flipcrossings{1,3} 
        \strand[white, double=black, thick, double distance=1pt] (-1,-1) 
        to [out=left, in=down] (-2,0)
        to [out=up, in=left] (-0.5,2)
        to [out=right, in=up] (0.25,1.5)
        to [out=down, in=up] (-1.5,0)
        to [out=down, in=left] (-1,-0.5); 
        \strand[white, double=black, thick, double distance=1pt] (-1,-0.5) 
        to [out=right, in=left] (0,-1)
        to [out=right, in=left] (1,-0.5); 
        \strand[white, double=black, thick, double distance=1pt] (1,-0.5) 
        to [out=right, in=down] (1.5,0)
        to [out=up, in=down] (-0.25,1.5)
        to [out=up, in=left] (0.5,2)
        to [out=right, in=up] (2,0) 
        to [out=down, in=right] (1,-1); 
        \strand[white, double=black, thick, double distance=1pt] (1,-1) 
        to [out=left, in=right] (0,-0.5)
        to [out=left, in=right] (-1,-1); 
        \node[] at (0,-2) {\Large $T^-_{+1}=S$}; 
    \end{knot} 
\end{tikzpicture} 
\quad 
\begin{tikzpicture} []
    \begin{knot} [] 
    \flipcrossings{1} 
        \strand[white, double=black, thick, double distance=1pt] (-1,-1) 
        to [out=left, in=down] (-2,0)
        to [out=up, in=left] (-0.5,2)
        to [out=right, in=up] (0.25,1.5)
        to [out=down, in=up] (-1.5,0)
        to [out=down, in=left] (-1,-0.5); 
        \strand[white, double=black, thick, double distance=1pt] (-1,-0.5) 
        to [out=right, in=left] (1,-0.5); 
        \strand[white, double=black, thick, double distance=1pt] (1,-0.5) 
        to [out=right, in=down] (1.5,0)
        to [out=up, in=down] (-0.25,1.5)
        to [out=up, in=left] (0.5,2)
        to [out=right, in=up] (2,0) 
        to [out=down, in=right] (1,-1); 
        \strand[white, double=black, thick, double distance=1pt] (1,-1) 
        to [out=left, in=right] (-1,-1); 
        \node[] at (0,-2) {\Large $T^-_0=U$}; 
    \end{knot} 
\end{tikzpicture} 
\quad 
\begin{tikzpicture} [] 
    \begin{knot} [] 
    \flipcrossings{1,4} 
        \strand[white, double=black, thick, double distance=1pt] (-1,-1) 
        to [out=left, in=down] (-2,0)
        to [out=up, in=left] (-0.5,2)
        to [out=right, in=up] (0.25,1.5)
        to [out=down, in=up] (-1.5,0)
        to [out=down, in=left] (-1,-0.5); 
        \strand[white, double=black, thick, double distance=1pt] (-1,-0.5) 
        to [out=right, in=left] (0,-1)
        to [out=right, in=left] (1,-0.5); 
        \strand[white, double=black, thick, double distance=1pt] (1,-0.5) 
        to [out=right, in=down] (1.5,0)
        to [out=up, in=down] (-0.25,1.5)
        to [out=up, in=left] (0.5,2)
        to [out=right, in=up] (2,0) 
        to [out=down, in=right] (1,-1); 
        \strand[white, double=black, thick, double distance=1pt] (1,-1) 
        to [out=left, in=right] (0,-0.5)
        to [out=left, in=right] (-1,-1); 
        \node[] at (0,-2) {\Large $T^-_{-1}=-T$};
    \end{knot} 
\end{tikzpicture}
\quad 
\begin{tikzpicture} [] 
    \begin{knot} [] 
    \flipcrossings{1,4,6} 
        \strand[white, double=black, thick, double distance=1pt] (-1,-1) 
        to [out=left, in=down] (-2,0)
        to [out=up, in=left] (-0.5,2)
        to [out=right, in=up] (0.25,1.5)
        to [out=down, in=up] (-1.5,0)
        to [out=down, in=left] (-1,-0.5); 
        \strand[white, double=black, thick, double distance=1pt] (-1,-0.5) 
        to [out=right, in=left] (-0.5,-1)
        to [out=right, in=left] (0,-0.5)
        to [out=right, in=left] (0.5,-1)
        to [out=right, in=left] (1,-0.5); 
        \strand[white, double=black, thick, double distance=1pt] (1,-0.5) 
        to [out=right, in=down] (1.5,0)
        to [out=up, in=down] (-0.25,1.5)
        to [out=up, in=left] (0.5,2)
        to [out=right, in=up] (2,0) 
        to [out=down, in=right] (1,-1); 
        \strand[white, double=black, thick, double distance=1pt] (1,-1) 
        to [out=left, in=right] (0.5,-0.5)
        to [out=left, in=right] (0,-1)
        to [out=left, in=right] (-0.5,-0.5)
        to [out=left, in=right] (-1,-1); 
        \node[] at (0,-2) {\Large $T^-_{-2}=-R$};
    \end{knot} 
\end{tikzpicture} }
        \caption{The negatively clasped $t$-twist knots $T^-_t$ for $t\in\{+1,0,-1,-2\}$. } 
        \label{subfigure:twist-} 
    \end{subfigure} 
    \caption{The twist knots for which all non-integral slopes are characterising. } 
    \label{figure:twist} 
\end{figure} 

\begin{theorem}[JSJ decomposition theorem \cite{js,j}]
    Let $M$ be a compact orientable irreducible 3-manifold with empty or toroidal boundary. Then there exists a minimal finite collection of essential tori in $M$ (unique up to isotopy) called the \emph{JSJ tori} such that each of the \emph{JSJ pieces} formed by cutting $M$ along these tori is either atoroidal or a Seifert fibred space. 
\end{theorem} 

When $M$ is the complement $\mathbb{S}^3_K$ of a knot $K\subset\mathbb{S}^3$, we can determine exactly what the JSJ pieces look like. Crucially, the JSJ tori correspond to those appearing naturally in the satellite construction. The \emph{outermost} JSJ piece of $\mathbb{S}^3_K$ is the unique one containing $\partial \mathbb{S}^3_K$. If this is hyperbolic, then $K$ is called a \emph{knot of hyperbolic type}. 

\begin{theorem}[\cite{budney}]
\label{theorem:jsj}
    Let $K\subset\mathbb{S}^3$ be a knot. Then the JSJ pieces of $\mathbb{S}^3_K$ can be classified into the following types. 
    \begin{enumerate}
        \item A Seifert fibred space over the disc orbifold $\mathbb{D}^2(a,b)$, i.e. a \emph{torus knot complement}, $\mathbb{S}^3_{T_{a,b}}$.
        \begin{itemize} 
            \item If this is the outermost JSJ piece, then $K$ is an \emph{$(a,b)$-torus knot}.
            \item Regular fibres have slope $ab/1$ on $\partial \mathbb{S}^3_K$. 
        \end{itemize} 
        \item A Seifert fibred space over the annulus orbifold $\mathbb{A}^2(s)$, i.e. a \emph{cable space}, $\mathbb{S}^3_{C_{r,s}}$. 
        \begin{itemize} 
            \item If this is the outermost JSJ piece, then $K$ is an \emph{$(r,s)$-cable knot}.
            \item Regular fibres have slope $rs/1$ on $\partial \mathbb{S}^3_K$ and slope $s/r$ on the other boundary torus. 
        \end{itemize} 
        \item A Seifert fibred space over the planar surface $\Sigma=\mathbb{S}^2\setminus\sqcup_{i=1}^{m\geq3}\mathbb{D}^2$, i.e. a \emph{composing space}. 
        \begin{itemize} 
            \item If this is the outermost JSJ piece, then $K$ is a \emph{composite knot}.
            \item Regular fibres have slope $1/0$ on $\partial \mathbb{S}^3_K$ and slope $0/1$ on the other boundary tori. 
        \end{itemize} 
        \item A hyperbolic 3-manifold that is homeomorphic to a \emph{hyperbolic link complement}, $\mathbb{S}^3_L$, such that removing one specific component $L_0$ of the link $L\subset\mathbb{S}^3$ leaves an unlink (possibly the empty link). 
        \begin{itemize} 
            \item If this is the outermost JSJ piece, then $K$ is a \emph{knot of hyperbolic type}. 
            \item Moreover, $K$ corresponds to this specific component $L_0 \subset L$.
        \end{itemize} 
    \end{enumerate} 
\end{theorem} 

In particular, each JSJ piece of a knot complement $\mathbb{S}^3_K$ is homeomorphic to a link complement $\mathbb{S}^3_L$; as decribed in \cite{budney}, this link $L$ is unique up to isotopy, given the additional data of the gluing maps used in the satellite construction of $K$. Therefore we can measure slopes on each toroidal boundary component of a JSJ piece $\mathbb{S}^3_L$ with respect to the basis given by the meridian and longitude of the corresponding component of $L$.

\subsection{Dehn surgery} 

For the purpose of introducing our notation, we briefly recall the definition of \emph{Dehn surgery}. If $N$ is a compact orientable 3-manifold, then \emph{Dehn filling} describes the process of attaching a solid torus $\mathbb{S}^1 \times \mathbb{D}^2$ to a toroidal boundary component $\mathbb{T}^2 \subset \partial N$ via a homeomorphism $$\phi : \partial (\mathbb{S}^1 \times \mathbb{D}^2) \to \mathbb{T}^2 \subset \partial N, \quad \phi(\{pt\} \times \partial \mathbb{D}^2) = \gamma,$$ where $\gamma \subset \mathbb{T}^2$ is some simple closed curve. We may write $N(\gamma)$ for the resulting manifold. In particular, when $K$ in a nullhomologous knot inside a compact orientable 3-manifold $M$, we can glue the solid torus $\nu(K) \cong \mathbb{S}^1 \times \mathbb{D}^2$ back into the knot complement $N=M_K$ by a homeomorphism $$\phi: \partial \nu(K) \to \mathbb{T}^2 \subset \partial M_K, \quad \phi(\{pt\} \times \partial \mathbb{D}^2) = p\mu + q\lambda,$$ where $\mu$ and $\lambda$ denote the meridian and longitude of $\mathbb{T}^2$ induced by $K$ and $p, q \in\mathbb{Z}$ are some pair of coprime integers. The resulting manifold is determined by the slope $p/q$ of $\gamma = p\mu+q\lambda$, so we denote it by $M_K(p/q)$. 

This process can be extended to links $L \subset M$ in the obvious way: we first choose a meridian and longitude for each component of $L$ by considering them in isolation as knots in $M$, and then we can write $$M_{L_0,\ldots,L_{m-1}}(p_0/q_0, \ldots, p_{m-1}/q_{m-1})$$ for Dehn surgery on an $m$-component link $L=L_0\cup \ldots\cup L_{m-1}$. For an unfilled $i^\text{th}$ component, we may write $\ast$ instead of $p_i/q_i$, or omit the coefficient altogether when the meaning is clear from the context. In this article, whenever $L = L_0 \cup U^{m-1}$, an unspecified filling component will always be taken to mean $L_0$. We will be particularly interested in performing Dehn surgery on $2$-component links $P = Q \cup U$; the same convention then allows us to write $$\mathbb{S}^3_P(p/q) = \mathbb{S}^3_{Q,U}(p/q, \ast)$$ without specifying the filling component. 

There are several results concerning Dehn surgery on knots in $\mathbb{S}^3$ and $\mathbb{S}^1\times\mathbb{D}^2$ which we will use in this article. An important quantity to consider when comparing Dehn fillings is the \emph{distance} $\Delta$ between a pair of slopes, which is precisely their minimal intersection number when drawn as simple closed curves on $\mathbb{T}^2$. In particular, note that $\Delta(p/q,1/0)=|q|$: this is the reason behind the conditions on the denominator that appear in the following results. 

First, we consider surgeries on a hyperbolic knot in $\mathbb{S}^3$. 

\begin{theorem}
\label{theorem:hyperbolic_knot}
    Let $K\subset\mathbb{S}^3$ be hyperbolic. If $|q|\geq2$, then $\mathbb{S}^3_K(p/q)$ is irreducible. If $|q|\geq3$, then $\mathbb{S}^3_K(p/q)$ is atoroidal. If $|q|\geq 9$, then $\mathbb{S}^3_K(p/q)$ is hyperbolic. 
\end{theorem} 

This follows from the maximal distances between exceptional fillings \cite{gl-red, gl-tor, lm}.

A similar result holds when the ambient manifold is the complement of an $(m-1)$-component unlink $U^{m-1}$ (where $m \geq 2$).

\begin{theorem}
\label{theorem:hyperbolic_pattern}
    Let $L_0 \subset\mathbb{S}^3_{U^{m-1}}$ be hyperbolic and write $L = L_0 \cup U^{m-1}$. If $|q|\geq2$, then $\mathbb{S}^3_L(p/q)$ is irreducible with incompressible boundary. If $|q|\geq3$, then $\mathbb{S}^3_L(p/q)$ is hyperbolic. 
\end{theorem} 

This follows from obstructing essential surfaces inside $\mathbb{S}^3_L(p/q)$; see \cite[Theorem~2.8]{lackenby}. 

Moving on to torus knots, we summarise the complete classification of surgeries.

\begin{theorem}[\cite{moser}]
\label{theorem:torus_knot}
    Let $K=T_{r,s}$ be the $(r,s)$-torus knot in $\mathbb{S}^3$. Then 
    $$\mathbb{S}^3_{T_{r,s}}(p/q) \cong
        \begin{cases}
            L(r,s) \# L(s,r) & p=rs, q=1, \\
            L(|p|,qs^2) & p=qrs\pm1, \\
            M & \text{otherwise}, \\
        \end{cases}
    $$
    where $M$ is a Seifert fibred space over $\mathbb{S}^2$ with three exceptional fibres of orders $|r|$, $|s|$ and $|p-qrs|$.
\end{theorem}

Viewing the torus knot $T_{r,s}$ as the $(r,s)$-cable of the unknot, the following result can be considered a generalisation of this.

\begin{theorem}[\cite{gordon}]
\label{theorem:cable_knot}
    Let $K=C_{r,s}(J)$ be the $(r,s)$-cable of a knot $J$ in $\mathbb{S}^3$. Then 
    $$\mathbb{S}^3_{C_{r,s}(J)}(p/q) \cong
        \begin{cases}
            \mathbb{S}^3_{J}(r/s) \# L(s,r) & p=rs, q=1, \\
            \mathbb{S}^3_{J}(p/qs^2) & p=qrs\pm1, \\
            \mathbb{S}^3_{J} \cup_{\mathbb{T}^2} M' & \text{otherwise}, \\
        \end{cases}
    $$
    where $M'$ is a Seifert fibred space over $\mathbb{D}^2$ with two exceptional fibres of orders $|s|$ and $|p-qrs|$.
\end{theorem}

We also note the following advancement towards the cosmetic surgery conjecture. 

\begin{theorem}[\cite{ni-wu}]
\label{theorem:csc}
    Let $K\subset \mathbb{S}^3$ and suppose that $\mathbb{S}^3_K(p/q) \cong \mathbb{S}^3_K(p'/q')$. Then $p/q=\pm p'/q'$. 
\end{theorem} 

Combining the previous two results gives the following corollary, which will prove to be very useful when considering cable knots. 

\begin{corollary} 
\label{corollary:csc}
    Let $K'=C_{r,s}(K)$ be an $(r,s)$-cable of $K$ and let $p/q$ be a slope with \mbox{$|p-qrs|=1$}. 
    
    Then there is no orientation-preserving homeomorphism $\mathbb{S}^3_K(p/q)\cong\mathbb{S}^3_{K'}(p/q)$. 
    \begin{proof} 
        Suppose otherwise, so that $\mathbb{S}^3_K(p/q)\cong\mathbb{S}^3_{K'}(p/q)\cong\mathbb{S}^3_K(p/qs^2)$ (as $|p-qrs|=1$). Then the slopes $p/q$ and $p/qs^2$ form a pair of cosmetic surgeries on $K$. By Theorem \ref{theorem:csc}, this is impossible because $|s|>1 \implies p/q \neq \pm p/qs^2$. 
    \end{proof} 
\end{corollary} 

We will next see how we can combine these results with Theorem \ref{theorem:jsj} in order to investigate and compare the manifolds produced by performing Dehn surgery on knots.

\section{JSJ decompositions} 
\label{section:jsj} 

In this section, we will use JSJ decompositions to gather crucial information about when we can have an orientation-preserving homeomorphism $\mathbb{S}^3_K(p/q)\cong\mathbb{S}^3_{K'}(p'/q')$ between manifolds obtained by Dehn surgery on knots $K$ and $K'$. In particular, we will fix $K$ to be either a hyperbolic knot or a satellite knot $P(J)$ by a hyperbolic pattern $P = Q \cup U$ and show that this forces the other knot to be of a very similar form. For the time being, we allow the surgery slopes $p/q$ and $p'/q'$ to differ, yielding more general results. Notation of the form $\mathbb{S}^3_P(p/q)$ always corresponds to filling the component corresponding to $Q$ rather than $U$.

\subsection{Creation and destruction} 

Given a knot $K$, it is very easy to describe the JSJ decomposition of the complement $\mathbb{S}^3_K$; however, after Dehn filling to produce $\mathbb{S}^3_K(p/q)$ (which remains irreducible for $|q|\geq2$ by \cite{gl-red}), it is not always the case that the JSJ tori are formed precisely from those of $\mathbb{S}^3_K$. There are two ways in which the JSJ decomposition can change: creation of new JSJ tori and destruction of existing JSJ tori. We will see that the possibility of creation of new JSJ tori can be immediately ruled out by implementing the assumption that $|q|\geq3$, and that destruction of existing JSJ tori only occurs in one very specific situation. Note that, in order for a JSJ torus to disappear during Dehn filling, the process must cause it to either bound some solid torus $\mathbb{S}^1 \times \mathbb{D}^2$, cobound (with another JSJ torus) some thickened torus $\mathbb{T}^2 \times \mathbb{I}$ or have Seifert fibred spaces on either side which match up to merge into the same JSJ piece.  

The following result gives a precise description of the JSJ decomposition of $\mathbb{S}^3_K(p/q)$, demonstrating that at most one JSJ torus can disappear during filling and that it must come from a cable space (see also \cite{lackenby} and \cite{sorya}). 

\begin{proposition}
\label{proposition:surgered} 
    Let $K$ be a knot with outermost JSJ piece $Y$ and let $p/q$ be a slope with $|q|\geq3$. 

    Then there is a JSJ piece of $\mathbb{S}^3_K(p/q)$ which contains the surgery curve and it takes one of the following forms: 
    \begin{enumerate}[(i)] 
        \item $Y(p/q)$; 
        \item $\hat{Y}(p/qs^2) \iff Y$ is an $(r,s)$-cable space next to a JSJ piece $\hat{Y}$ and $|p-qrs|=1$. 
    \end{enumerate} 
    \begin{proof}
        Let $Y$ be the outermost JSJ piece of $\mathbb{S}^3_K$. We will use the classification of JSJ pieces in Theorem \ref{theorem:jsj} to consider each possibility in turn.  

        Case (1) arises when $Y=\mathbb{S}^3_{T_{a,b}}$ for a torus knot $K=T_{a,b}$. Since $|q|\geq2$, we can ignore the $|p-qrs|=0$ case of Theorem \ref{theorem:torus_knot} and observe that $Y(p/q)=\mathbb{S}^3_K(p/q)$ is irreducible and atoroidal. We deduce that $Y(p/q)$ is still a single Seifert fibred JSJ piece.  

        Case (4) happens when $Y$ is the complement of a certain type of hyperbolic link. When $|\partial Y|>1$, the assumption $|q| \geq 3$ tells us that $Y(p/q)$ is hyperbolic and so it remains a single JSJ piece. In the case where $|\partial Y|=1$ and $K$ is a hyperbolic knot, a stronger assumption of $|q|\geq9$ is required to ensure that $Y(p/q)=\mathbb{S}^3_K(p/q)$ remains hyperbolic; the $|q|\geq3$ assumption allows the possibility that it is Seifert fibred, but this does not affect the conclusion as there are no other JSJ pieces over which this structure could extend. 

        Case (3) takes place when $Y$ is a composing space. Regular fibres have meridional slope on the boundary $\partial \mathbb{S}^3_K$, so $|q|\geq1$ implies that the Seifert fibration extends to $Y(p/q)$. Since this adds an exceptional fibre of order $|q|\geq2$, the boundary tori do not compress; moreover, the Seifert fibred structure is otherwise unchanged and unique, so it cannot extend over any adjacent JSJ pieces for precisely the reason that it did not before filling. Hence $Y(p/q)$ is a single Seifert fibred JSJ piece.  

        Case (2) occurs when $Y=\mathbb{S}^3_{C_{r,s}}$ is a cable space and $K=C_{r,s}(J)$ is a cable knot. Since $|q|\geq2$, we can ignore the $|p-qrs|=0$ case of Theorem \ref{theorem:cable_knot}. When $|p-qrs|>1$, the Seifert fibred structure on $Y$ extends to $Y(p/q)$ over the new solid torus, introducing a new exceptional fibre of order $|p-qrs|>1$ (which prevents compression of the boundary tori). As the Seifert fibred structure is unique (unless $|s|=|p-qrs|=2$: see \cite[Proposition~3.6]{sorya} for a more detailed discussion of this), there is no way for it to extend over any other attached JSJ pieces, and so $Y(p/q)$ is a single Seifert fibred JSJ piece. 
        
        By Theorem \ref{theorem:cable_knot}, compression happens precisely when $|p-qrs|=1$, in which case we have $\mathbb{S}^3_K(p/q)\cong\mathbb{S}^3_J(p/qs^2)$. However, letting $\hat{Y}$ be the outermost JSJ piece of $\mathbb{S}^3_J$, we will show that the same thing cannot happen again. If $J=C_{r',s'}(J')$ and $|p-qr's's^2|=1$, then 
        $$0 \leq 4|r-r's's| \leq |q||s||r-r's's| = |qrs - qr's's^2| \leq |p-qrs| + |p-qr's's^2| =2$$ 
        and we obtain $r/s=r's' \in \mathbb{Z}$, a contradiction. Hence the cable cannot be iterated and there is no further compression of tori. Since $|qs^2|>|q|\geq3$ and the JSJ piece $\hat{Y}$ is not a compressing cable space, we can simply apply the same arguments as we did for each of the other possibilities for $\hat{Y}$ to deduce that the JSJ piece of $\mathbb{S}^3_K(p/q)$ containing the surgery curve in this case is precisely $\hat{Y}(p/qs^2)$.  
    \end{proof}
\end{proposition}

We will often use the notation $\tilde{Y}(\tilde{p}/\tilde{q})$ to represent the JSJ piece containing the surgery curve. This encompasses both cases and reflects the fact that the JSJ piece $\tilde{Y}$ will be the one which is most important in our discussion. Moreover, the slope $\tilde{p}/\tilde{q}$ satisfies $|\tilde{q}|\geq|q|$, so it is convenient to note that lower bounds on the denominator automatically apply.

\subsection{Hyperbolic knots} 

Now we will suppose that there is an orientation-preserving homeomorphism $\mathbb{S}^3_K(p/q) \cong \mathbb{S}^3_{K'}(p'/q')$ and use what we know about JSJ decompositions to compare $K$ and $K'$.  

Let's begin by considering at the scenario in which $K$ is a fixed hyperbolic knot. Proposition \ref{proposition:surgered} provides enough information to determine the basic type of the unknown knot $K'$. 

\begin{proposition} 
\label{proposition:knot}
    Let $K$ be a hyperbolic knot. Suppose that there exists an orientation-preserving homeomorphism $\mathbb{S}^3_K(p/q) \cong \mathbb{S}^3_{K'}(p'/q')$ for some knot $K'$ and some slopes $p/q$, $p'/q'$ with $|q|\geq9$ and $|q'|\geq3$.  

    Then one of the following holds: 
    \begin{enumerate}[(i)]
        \item $K'$ is hyperbolic; 
        \item $K'=C_{r,s}(\hat{K})$ is the $(r,s)$-cable of a hyperbolic knot $\hat{K}$, $|p'-q'rs|=1$ and there is an orientation-preserving homeomorphism $\mathbb{S}^3_K(p/q) \cong \mathbb{S}^3_{\hat{K}}(p'/q's^2)$.  
    \end{enumerate}
    \begin{proof} 
        Since $K$ is a hyperbolic knot and $|q|\geq9$, the filling $\mathbb{S}^3_K(p/q)$ remains hyperbolic by Theorem \ref{theorem:hyperbolic_knot} and contains no JSJ tori. We have an orientation-preserving homeomorphism $\mathbb{S}^3_K(p/q)\cong\mathbb{S}^3_{K'}(p'/q')$ for some other knot $K'$ whose complement $\mathbb{S}^3_{K'}$ has outermost JSJ piece $Y'$. By Proposition \ref{proposition:surgered}, the JSJ piece of $\mathbb{S}^3_{K'}(p'/q')$ containing the surgery curve -- in this case, the only JSJ piece -- is of the form $\tilde{Y}(\tilde{p}/\tilde{q})$; since this is hyperbolic, $\tilde{Y}$ must also be hyperbolic. Since $|\partial \tilde{Y}|=1$, we must have $\tilde{Y}\cong\mathbb{S}^3_{\tilde{K}}$ for some hyperbolic knot $\tilde{K}$. Reconciling $\tilde{Y}$ with $Y'$ and $\hat{Y}$ as in Proposition \ref{proposition:surgered} gives the two cases. 
    \end{proof} 
\end{proposition} 

We will write $(\tilde{K},\tilde{p}/\tilde{q})$ to mean either (i) $(K',p'/q')$ or (ii) $(\hat{K},p'/q's^2)$, depending on which case of Proposition \ref{proposition:knot} we are in. In both situations, $\tilde{K}$ is hyperbolic and there is an orientation-preserving homeomorphism \mbox{$\mathbb{S}^3_K(p/q) \cong \mathbb{S}^3_{\tilde{K}}(\tilde{p}/\tilde{q})$}.

\subsection{Hyperbolic patterns} 

We would like an analogous result for satellite knots which can be expressed as $K=P(J)$ for a hyperbolic pattern $P = Q \cup U$. However, as $\mathbb{S}^3_K(p/q)$ will now comprise more than one JSJ piece, there is an extra obstacle of checking that the orientation-preserving homeomorphism $\mathbb{S}^3_K(p/q) \cong \mathbb{S}^3_{K'}(p'/q')$ restricts to one between the JSJ pieces containing the surgery curves. 

Recall the following result of Gordon. 

\begin{lemma}[{\cite[Lemma~3.3]{gordon}}]
    \label{lemma:gordon} 
    Let $P = Q \cup U$ be any pattern with winding number $w$. 
    
    Then $H_1(\mathbb{S}^3_P(p/q)) \cong \mathbb{Z} \oplus \mathbb{Z}/d\mathbb{Z}$, where $d=\gcd(p,w)$, and the kernel of $H_1(\partial\mathbb{S}^3_P(p/q)) \to H_1(\mathbb{S}^3_P(p/q))$ is generated by
    $$ 
    \begin{cases}
        (p/d)\lambda_U + (qw^2/d) \mu_U & \text{if} \ w\neq0, \\
        \lambda_U & \text{if} \ w=0. \\
    \end{cases}
    $$
    \begin{proof} 
        Observe that $H_1(\mathbb{S}^3_P)\cong\mathbb{Z}\oplus\mathbb{Z}$ is freely generated by the meridians $\mu_Q$ and $\mu_U$. Moreover, note that $\lambda_Q = w\mu_U$ and $\lambda_U = w\mu_Q$ in $H_1(\mathbb{S}^3_P)$. After Dehn filling the component corresponding to $Q$ along slope $p/q$, we see that $p\mu_Q+q\lambda_Q=0$ in $H_1(\mathbb{S}^3_P(p/q))$. When $w=0$, $p\mu_Q$ vanishes; since $\gcd(p,0)=|p|$, we obtain $H_1(\mathbb{S}^3_P(p/q))\cong\mathbb{Z}\oplus\mathbb{Z}/|p|\mathbb{Z}$, and we still have $\lambda_U=0$ so the kernel is generated by $\lambda_U$. When $w\neq 0$, we have $p\mu_Q+qw\mu_U=0$, so $H_1(\mathbb{S}^3_P(p/q))\cong\mathbb{Z}\oplus\mathbb{Z}/d\mathbb{Z}$; observing also that $p\lambda_U+qw^2\mu_U=0$, it is simple to deduce that the kernel is generated by $(p/d)\lambda_U + (qw^2/d) \mu_U$, as required. 
    \end{proof}
\end{lemma} 

Consequently, $\mathbb{S}^3_P(p/q)$ can only be homeomorphic to a knot complement when $\gcd(p,w)=1$. This is significant when it is a single JSJ piece because by Theorem \ref{theorem:jsj}, the only possible JSJ pieces of a satellite knot complement with a single boundary component take the form of simple knot complements. 

We will now see that the homeomorphism $\mathbb{S}^3_K(p/q) \cong \mathbb{S}^3_{K'}(p'/q')$ restricts to a \emph{slope-preserving} homeomorphism $\mathbb{S}^3_P(p/q) \cong \mathbb{S}^3_{P'}(p'/q')$ between the filled pattern pieces (meaning that the meridian and longitude -- and hence all other slopes -- on the boundary torus are identified). 

\begin{proposition} 
\label{proposition:pattern}
    Let $K=P(J)$ be a satellite of a knot $J$ by a hyperbolic pattern \mbox{$P = Q \cup U$} with winding number $w$. Suppose that there exists an orientation-preserving homeomorphism \mbox{$\mathbb{S}^3_K(p/q)\cong\mathbb{S}^3_{K'}(p'/q')$} for some knot $K'$ and some slopes $p/q, p'/q'$ with $\gcd(p,w)\neq1$ and $|q|,|q'|\geq3$. 

    Then $K'=P'(J)$ is a satellite of the same knot $J$ by a pattern $P'$ and there is a slope-preserving homeomorphism $\mathbb{S}^3_P(p/q)\cong\mathbb{S}^3_{P'}(p'/q')$, where one of the following holds: 
    \begin{enumerate}[(i)] 
        \item $P' = Q' \cup U'$ is hyperbolic; 
        \item $P' = C_{r,s}(\hat{P})$ is the $(r,s)$-cable of a hyperbolic pattern $\hat{P} = \hat{Q} \cup \hat{U}$, $|p'-q'rs|=1$ and there is a slope-preserving homeomorphism $\mathbb{S}^3_P(p/q) \cong \mathbb{S}^3_{\hat{P}}(p'/q's^2)$. 
    \end{enumerate}
    \begin{proof}
        Since the outermost JSJ piece $Y=\mathbb{S}^3_P$ of $\mathbb{S}^3_K$ is hyperbolic and $|q|\geq3$, Proposition \ref{proposition:surgered} tells us that the JSJ piece containing the surgery curve is $Y(p/q)=\mathbb{S}^3_P(p/q)$. Consider the JSJ piece of $\mathbb{S}^3_{K'}$ containing the surgery curve, which is of the form $\tilde{Y}(\tilde{p}/\tilde{q})$ by Proposition \ref{proposition:surgered}. The homeomorphism $\mathbb{S}^3_K(p/q)\cong\mathbb{S}^3_{K'}(p'/q')$ preserves the JSJ decompositions, so there must be at least one JSJ torus present in $\mathbb{S}^3_{K'}(p'/q')$. If $K'$ were a simple knot, so that $\tilde{Y}=\mathbb{S}^3_{K'}$, then $\tilde{Y}(\tilde{p}/\tilde{q})=\mathbb{S}^3_{K'}(p'/q')$ and there are no JSJ tori. We deduce that $K'$ is not a simple knot and therefore must be a satellite knot. 
        
        Since $\gcd(p,w)\neq1$, Lemma \ref{lemma:gordon} implies that the JSJ piece $\mathbb{S}^3_P(p/q)$ cannot be homeomorphic to a knot complement. By Theorem \ref{theorem:jsj}, the only possible JSJ pieces of $\mathbb{S}^3_{K'}$ with a single boundary component are simple knot complements, so the homeomorphism $\mathbb{S}^3_K(p/q)\cong\mathbb{S}^3_{K'}(p'/q')$ must take $\mathbb{S}^3_P(p/q)$ to the JSJ piece $\tilde{Y}(\tilde{p}/\tilde{q})$ that was created during Dehn filling. Therefore we also have an orientation-preserving homeomorphism from the companion knot complement $\mathbb{S}^3_J$ to some other knot complement $\mathbb{S}^3_{J'}\subset\mathbb{S}^3_{K'}$ which is not affected by the filling. Since knots are determined by their complements \cite{gl-knots}, it follows that $J=J'$ and additionally that the homeomorphism $\mathbb{S}^3_J \cong \mathbb{S}^3_{J'}$ is slope-preserving. We can now express the satellite knot as $K'=P'(J)$ for a pattern $P'$ such that $\tilde{Y}(\tilde{p}/\tilde{q})=\mathbb{S}^3_{P'}(p'/q')$ and there is a slope-preserving homeomorphism $\mathbb{S}^3_P(p/q)\cong\mathbb{S}^3_{P'}(p'/q')$. 

        By Proposition \ref{proposition:surgered}, $\mathbb{S}^3_{P'}(p'/q')=\tilde{Y}(\tilde{p}/\tilde{q})$ for some JSJ piece $\tilde{Y}\subset\mathbb{S}^3_{K'}$; since this is hyperbolic by Theorem \ref{theorem:hyperbolic_pattern}, $\tilde{Y}$ must also be hyperbolic. Since $|\partial\tilde{Y}|=2$, we must have $\tilde{Y} \cong \mathbb{S}^3_{\tilde{P}}$ for some pattern $\tilde{P} = \tilde{Q} \cup \tilde{U}$. The two options in Proposition \ref{proposition:surgered} yield each of the possible pattern types. 
    \end{proof} 
\end{proposition} 

We will henceforth use $(\tilde{P},\tilde{p}/\tilde{q})$ to denote either (i) $(P',p'/q')$ or (ii) $(\hat{P},p'/q's^2)$, depending on which case of Proposition \ref{proposition:pattern} we are in. Thus $\tilde{P} = \tilde{Q} \cup \tilde{U}$ is a hyperbolic pattern and there is a slope-preserving homeomorphism $\mathbb{S}^3_P(p/q) \cong \mathbb{S}^3_{\tilde{P}}(\tilde{p}/\tilde{q})$.

\subsection{Pattern properties} 

Before moving onto the core of the proofs of our main theorems, we take a closer look at what a slope-preserving homeomorphism $\mathbb{S}^3_P(p/q) \cong \mathbb{S}^3_{P'}(p'/q')$ means for the patterns. 

We begin by looking at the special case of patterns with winding number zero. 

\begin{lemma} 
\label{lemma:winding} 
    Let $P = Q \cup U$ be any pattern with winding number $w=0$. 
    
    If there exists a slope-preserving homeomorphism $\mathbb{S}^3_P(p/q)\cong\mathbb{S}^3_{P'}(p'/q')$ for some pattern \mbox{$P' = Q' \cup U'$} and some slopes $p/q,p'/q'$, then $P'$ has winding number $w'=0$.  
    
    If $P'=C_{r,s}(\hat{P})$ for some pattern $\hat{P} = \hat{Q} \cup \hat{U}$ and $|p'-q'rs|=1$, then $\hat{P}$ has winding number $\hat{w}=0$. 
    \begin{proof} 
        Any homeomorphism should take nullhomologous curves to nullhomologous curves. When $w=0$, the nullhomologous curve is precisely $\lambda_U$ by Lemma \ref{lemma:gordon}. A slope-preserving homeomorphism should take $\lambda_U$ to $\lambda_{U'}$. Again, by analysing the cases in Lemma \ref{lemma:gordon}, this forces $w'=0$.  
        
        If $P'=C_{r,s}(\hat{P})$ and $|p'-q'rs|=1$, then the cable space compresses and we have a slope-preserving homeomorphism $\mathbb{S}^3_P(p/q) \cong \mathbb{S}^3_{\hat{P}}(p'/q's^2)$. We conclude by the same argument that $\hat{w}=0$. 
    \end{proof} 
\end{lemma} 

We now set $p'/q'=p/q$ and use the Dehn surgery characterisation of the unknot $U$, trefoils $\pm T$, figure eight knot $S$ and $\pm 5_2$ knots $\pm R$ to gain information about any pattern which produces one of these knots when applied to the unknot. 

\begin{lemma} 
\label{lemma:knotting} 
    Let $P = Q \cup U$ be a pattern with $Q \in\{U,\pm T,S,\pm R\}$ in $\mathbb{S}^3$. 
    
    If there exists a slope-preserving homeomorphism $\mathbb{S}^3_P(p/q)\cong\mathbb{S}^3_{P'}(p/q)$ for some pattern $P' = Q' \cup U'$ and some non-integral slope $p/q$, then $Q=Q'$ in $\mathbb{S}^3$.  
    
    If $P'=C_{r,s}(\hat{P})$ for some pattern $\hat{P} = \hat{Q} \cup \hat{U}$ and $|p-qrs|=1$, then $Q\in\{U,\pm T$\} and $\hat{Q}=U$ in $\mathbb{S}^3$. 
    \begin{proof} 
        Apply both patterns $P$ and $P'$ to the unknot by gluing a solid torus to the exterior of each filled pattern piece in the standard way. This yields a homeomorphism $\mathbb{S}^3_Q(p/q) \cong \mathbb{S}^3_{Q'}(p/q)$. However, we know that every slope $p/q$ is characterising for $Q$ -- by \cite{kmos} for $U$, by \cite{os} for $\pm T$ and $S$ and by \cite{bs} for $\pm R$ (provided $|q|\geq2$). Therefore $Q'$ must be isotopic to the same knot in $\mathbb{S}^3$. 
        
        If $P'=C_{r,s}(\hat{P})$, then $Q = Q' = C_{r,s}(\hat{Q})$ in $\mathbb{S}^3$. The only way for this to hold when $Q=U$ is if $\hat{Q}=U$ and $|r|=1$. When $Q=\pm T$, we can have $\hat{Q}=U$ and $(r,s)\in\{(\pm2,3),(\pm3,2)\}$. We cannot have $Q=S$, nor $Q=\pm R$, because there is no way to express these hyperbolic knots as non-trivial cables. 
    \end{proof} 
\end{lemma}

\section{Minimal geodesics} 
\label{section:geodesics} 

The aim of this section is to show that if $L = L_0 \cup U^{m-1}$ and $\tilde{L} = \tilde{L}_0 \cup \tilde{U}^{m-1}$, then the existence of a slope-preserving homeomorphism $\mathbb{S}^3_L(p/q) \cong \mathbb{S}^3_{\tilde{L}}(\tilde{p}/\tilde{q})$ implies that $L_0=\tilde{L}_0$. We can do this by analysing shortest geodesics before and after Dehn filling. Setting $(L_0, \tilde{L}_0)$ to be $(K,\tilde{K})$ when $m=1$ and $(Q,\tilde{Q})$ when $m=2$, this will enable us to complete the proofs of Theorems \ref{theorem:new} and \ref{theorem:main}. 

\begin{theorem:new}
    Let $K$ be a hyperbolic knot. 

    Then every slope $p/q$ with $|q|\geq \max\{35, \mathfrak{q}(\mathbb{S}^3_K)\}$ is characterising for $K$. 
\end{theorem:new}

\begin{theorem:main}
    Let $K=P(J)$ be a satellite of a companion $J$ by a hyperbolic pattern $P = Q \cup U$ with winding number $w$.  

    Then every slope $p/q$ with $\gcd(p,w)\neq1$ and $|q|\geq \max\{35,\mathfrak{q}(\mathbb{S}^3_P)\}$ is characterising for $K$. 
\end{theorem:main} 

The argument goes as follows. We know that the slope-preserving homeomorphism \mbox{$\mathbb{S}^3_L(p/q) \cong \mathbb{S}^3_{\tilde{L}}(\tilde{p}/\tilde{q})$} is homotopic to an isometry by Mostow rigidity, and therefore preserves shortest geodesics. Furthermore, if $|q| \geq \max\{35, \mathfrak{q}(\mathbb{S}^3_L)\}$, then the shortest geodesic in $\mathbb{S}^3_L(p/q)$ is the core curve $c_{L_0}$ of the filling. This is because the geometry of $\mathbb{S}^3_L(p/q)$ away from the core curve looks very similar to that of $\mathbb{S}^3_L$, in which all geodesics are much longer by our hypothesis. By considering the corresponding geodesics in $\mathbb{S}^3_{\tilde{L}}(\tilde{p}/\tilde{q})$, it once again transpires that the shortest geodesic can only be the core curve $c_{\tilde{L}_0}$ of the filling. We now have a slope-preserving homeomorphism which takes $L_0$ to $\tilde{L}_0$; it follows that $L=\tilde{L}$.

\subsection{Hyperbolic geometry} 

Let $M$ be a compact orientable hyperbolic 3-manifold with toroidal boundary. Recall that each component of $\partial M$ has a neighbourhood called a \emph{cusp} which looks like $N \cong \mathbb{T}^2 \times [1,\infty)$. Expanding $N$ to the largest size for which the torus $\partial N$ is embedded rather than immersed gives rise to a \emph{maximal cusp} $N$, whose boundary $\partial N$ inherits a well-defined Euclidean metric and is the standard place to measure lengths of slopes on $\partial M$. 
\pagebreak

\begin{definition} 
    Let $M$ be a compact orientable hyperbolic 3-manifold with toroidal boundary $\partial M$. Let $\gamma$ be a slope on a component of $\partial M$ and let $N$ be the corresponding maximal cusp. 
    \begin{itemize} 
        \item The \emph{area} $A(\partial N)$ of $\partial N \cong\mathbb{T}^2$ is the usual Euclidean area. 
        \item The \emph{length} $l(\gamma)$ of $\gamma$ is the Euclidean length of a geodesic representative on $\partial N$. 
        \item The \emph{normalised length} $\hat{l}(\gamma)$ of $\gamma$ is $\hat{l}(\gamma) := l(\gamma) / \sqrt{A(\partial N)}$. 
    \end{itemize} 
\end{definition} 

\begin{lemma}
\label{lemma:facts}
    Let $M$ be a compact orientable hyperbolic 3-manifold with toroidal boundary $\partial M$. Let $\gamma$ be a slope on a component of $\partial M$ and let $N$ be the corresponding maximal cusp. 
    \begin{itemize} 
        \item The area $A(\partial N)$ of $\partial N \cong\mathbb{T}^2$ satisfies $A(\partial N) \geq 2\sqrt{3}$. 
        \item The length $l(\gamma)$ of $\gamma$ satisfies $l(\gamma)l(\mu) \geq \Delta(\gamma, \mu) \cdot A(\partial N)$. 
        \item The normalised length $\hat{l}(\gamma)$ of $\gamma$ satisfies $\hat{l}(\gamma)\hat{l}(\mu) \geq \Delta(\gamma, \mu)$. 
    \end{itemize} 
    \begin{proof}
        The first inequality follows from the lower bound of $\sqrt{3}$ on cusp volume \cite[Theorem~1.2]{ghmty}, which is half of cusp area. The others follow from a simple geometric argument: see for instance \cite[Lemma~8.1]{lackenby-notes}. 
    \end{proof}
\end{lemma} 

When $M$ is a link complement, we know that the canonical choice of meridian $\mu$ and longitude $\lambda$ on each component of $\partial M$ allows us to identify $\gamma=p\mu+q\lambda$ with $p/q \in \mathbb{Q}\cup\{\infty\}$. Recall that we have a notion of \emph{distance} between a pair of slopes $p/q$ and $r/s$, which corresponds to their geometric intersection number: $\Delta(p/q,r/s):=|ps-qr|$. In particular, substituting $\Delta(p/q,1/0)=|q|$ into Lemma \ref{lemma:facts} yields an explicit relationship between the length and denominator of the slope $p/q$. 

When a slope is sufficiently long, we can fill along it and maintain the hyperbolic structure on $M$. 

\begin{theorem}[6-theorem \cite{agol-6, lackenby-6}] 
\label{theorem:6}
    Let $M$ be a compact orientable hyperbolic 3-manifold with toroidal boundary $\partial M$. Let $\gamma$ be a slope on a component of $\partial M$ with length $l(\gamma)>6$. 

    Then $M(\gamma)$ is also hyperbolic. 
\end{theorem}

\subsection{Quantitative comparison} 

In order to formulate a quantitative argument, we will employ some results from \cite{fps} regarding lengths of geodesics. 

The first result allows us to compare the length of a sufficiently short geodesic $c \subset M$ with the normalised length of its meridian $\mu$ viewed in the complement $M_c$. 

\begin{theorem}[{\cite[Corollary~6.13]{fps}}]
\label{theorem:corollary6.13}
    Let $M$ be a complete hyperbolic 3-manifold with finite volume and let $c \subset M$ be a closed geodesic with complement $M_c$ and meridian $\mu$. Suppose that one or both of the following hypotheses is satisfied: 
    \begin{itemize} 
        \item $\hat{l}(\mu) \geq 7.823$ with respect to the complete hyperbolic structure on $M_c$; 
        \item $l(c) \leq 0.0996$ with respect to the complete hyperbolic structure on $M$. 
    \end{itemize} 
    
    Then these quantities satisfy the inequality $$\frac{2\pi}{\hat{l}(\mu)^2 + 16.17} < l(c) < \frac{2\pi}{\hat{l}(\mu)^2 - 28.78}.$$
\end{theorem} 

In particular, let $L=L_0 \cup U^{m-1}$ and set $M=\mathbb{S}^3_L(p/q)$ to be the result of Dehn surgery along the knot $L_0 \subset\mathbb{S}^3_{U^{m-1}}$. If we take our geodesic to be the core curve $c=c_{L_0} \subset \mathbb{S}^3_L(p/q)$ of the filling, then the meridian $\mu$ of $c$ viewed inside the manifold $M_c = \mathbb{S}^3_L(p/q)\setminus\nu(c) = \mathbb{S}^3_L$ corresponds to the surgery slope $p/q$ for $L_0$ which bounds a disc in $\mathbb{S}^3_L(p/q)$. This gives us a way to sandwich the length $l(c)$ of the core curve between upper and lower bounds which depend only on $\hat{l}(\mu)=\hat{l}(p/q)$. 

The next result allows us to analyse multiple short geodesics appearing simultaneously. 

\begin{theorem}[{\cite[Corollary~7.20]{fps}}]
\label{theorem:corollary7.20} 
    Let $M$ be a complete hyperbolic 3-manifold with finite volume and let $c \cup \gamma \subset M$ be a geodesic link in $M$. Suppose that $\max\{l(c),l(\gamma)\}\leq0.0735$.
    
    Then $\gamma$ is isotopic to a geodesic in the complete metric on the complement $M_c$ with length at most $1.9793 \cdot l(\gamma)$. 
\end{theorem} 

Intuitively, this means that if both $c = c_{L_0}$ and $\gamma$ are sufficiently short geodesics in $M=\mathbb{S}^3_L(p/q)$, then $\gamma$ must also have been very short in $M_{c}=\mathbb{S}^3_L$. Our strategy relies heavily on this result: we will show that if all geodesics in $\mathbb{S}^3_L$ were in fact very long, then it is impossible for any of them to have produced such a short geodesic $\gamma$ inside $\mathbb{S}^3_L$. This allows us to deduce that the core curve $c$ must be the unique shortest geodesic in $\mathbb{S}^3_L(p/q)$.

\subsection{Core curves correspond} 

Given our fixed hyperbolic link $L = L_0 \cup U^{m-1}$, we seek to demonstrate that taking \mbox{$|q| \geq \max\{35, \mathfrak{q}(\mathbb{S}^3_L)\}$} forces the shortest geodesic in $\mathbb{S}^3_L(p/q)$ to be the core curve $c = c_{L_0}$. We will do this by first ensuring that a certain condition on the normalised length $\hat{l}(p/q)$ of the slope is satisfied. This will then provide useful information about shortest geodesics. 

\begin{lemma} 
\label{lemma:normalised}
    Let $L = L_0 \cup U^{m-1}$ be a hyperbolic link and let $p/q$ be a slope on the boundary component of $\mathbb{S}^3_L$ corresponding to $L_0$ such that $|q| \geq q_0$. 
    
    Then the normalised length of the slope $p/q$ satisfies $\sqrt{6\sqrt{3}} \cdot \hat{l}(p/q) \geq q_0$. 
    \begin{proof} 
        By Lemma \ref{lemma:facts}, the normalised length $\hat{l}(p/q)$ of the slope $p/q$ satisfies the relationship 
        $$\hat{l}(p/q)\hat{l}(1/0) \geq \Delta(p/q,1/0) = |q|\geq q_0. $$  
        Since the trivial filling of $L_0 \subset \mathbb{S}^3_{U^{m-1}}$ produces something non-hyperbolic, Theorem \ref{theorem:6} tells us that $l(1/0) \leq 6$. We also have a universal lower bound $A(\partial N)\geq 2\sqrt{3}$ on cusp area from Lemma \ref{lemma:facts}, which implies that $\sqrt{2\sqrt{3}} \cdot \hat{l}(1/0) \leq 6$. Combining these inequalities gives $$\hat{l}(p/q) \geq \frac{|q|}{\hat{l}(1/0)} \geq \frac{\sqrt{2{\sqrt{3}}}}{6} \cdot |q| \geq \frac{q_0}{\sqrt{6\sqrt{3}}}$$ as required. 
        \qedhere 
    \end{proof} 
\end{lemma} 

We can thus bound the length of the core curve $c = c_{L_0}$ in the filling $\mathbb{S}^3_L(p/q)$ in the following way. Note that $L'$ is just some hyperbolic link (possibly, but not necessarily, $L$ itself): for now, the quantity $\mathfrak{q}(\mathbb{S}^3_{L'})$ is used in a purely numerical manner. 

\begin{lemma}
\label{lemma:numerical}
    Let $L= L_0 \cup U^{m-1}$ be a hyperbolic link and let $p/q$ be a slope on the boundary component of $\mathbb{S}^3_L$ corresponding to $L_0$ such that $|q| \geq \max \{35, \mathfrak{q}(\mathbb{S}^3_{L'})\}$ for some hyperbolic link $L'$. 
    
    Then the core curve $c = c_{L_0}$ of $\mathbb{S}^3_L(p/q)$ is a geodesic of length $l(c) < \min\{0.0706, {1.9793}^{-1}\cdot\sys(\mathbb{S}^3_{L'})\}$. 
    
    Moreover, any other geodesic $\gamma$ in $\mathbb{S}^3_L(p/q)$ with length $l(\gamma)\leq l(c)$ must also be isotopic to a geodesic in the hyperbolic link complement $\mathbb{S}^3_L$ with length strictly less than $\min\{0.14, \sys(\mathbb{S}^3_{L'})\}$. 
    \begin{proof} 
        As $|q| \geq 25$, Lemma \ref{lemma:normalised} implies that $\hat{l}(p/q) \geq 7.5832$, so $c$ is a geodesic \cite[Theorem 5.17]{fps}.
        Since $|q|\geq26$, Lemma \ref{lemma:normalised} tells us that $\hat{l}(p/q) \geq 7.823$, so we can apply Theorem \ref{theorem:corollary6.13} and Lemma \ref{lemma:normalised}, respectively, to obtain $$l(c) < \frac{2\pi}{\hat{l}(p/q)^2 - 28.78} \leq \frac{2\pi}{\frac{|q|^2}{6\sqrt{3}} - 28.78} $$ and then use the hypothesis $|q|\geq\max\{35, \mathfrak{q}(\mathbb{S}^3_{L'})\}$ to reach the desired result. 
        
        Now suppose that $\gamma$ is another geodesic in $\mathbb{S}^3_L(p/q)$ with length $l(\gamma) \leq l(c)$. By Theorem \ref{theorem:corollary7.20}, we can deduce that $\gamma$ must have been a geodesic with respect to the complete metric on the complement $\mathbb{S}^3_L(p/q)\setminus\nu(c) \cong \mathbb{S}^3_L$ with length at most $1.9793\cdot l(\gamma) < \min\{0.14, \sys(\mathbb{S}^3_{L'})\}$, as required. 
    \end{proof}
\end{lemma}

The following result ensures that when $|q|,|\tilde{q}|\geq\max\{35, \mathfrak{q}(\mathbb{S}^3_L)\}$ (note that $L'=L$ now) and there is a slope-preserving homeomorphism $\mathbb{S}^3_L(p/q)\cong\mathbb{S}^3_{\tilde{L}}(\tilde{p}/\tilde{q})$, we can deduce that the shortest geodesics are precisely the core curves. Note that when $m=1$ and $\mathbb{S}^3_L(p/q)\cong\mathbb{S}^3_{\tilde{L}}(p/q)$ is a homeomorphism of closed manifolds, the term ``slope-preserving'' can be interpreted simply as ``orientation-preserving''. 

\begin{proposition}
\label{proposition:35} 
    Let $L = L_0 \cup U^{m-1}$, $\tilde{L} = \tilde{L}_0 \cup \tilde{U}^{m-1}$ be hyperbolic links and let $p/q, \tilde{p}/\tilde{q}$ be slopes on the boundary components of $\mathbb{S}^3_L, \mathbb{S}^3_{\tilde{L}}$ corresponding to $L_0, \tilde{L}_0$, respectively, such that $|q|,|\tilde{q}| \geq \max\{35, \mathfrak{q}(\mathbb{S}^3_L)\}$. 
    
    If there exists a slope-preserving homeomorphism $\mathbb{S}^3_L(p/q)\cong\mathbb{S}^3_{\tilde{L}}(\tilde{p}/\tilde{q})$, then $L_0=\tilde{L}_0$. 
    \begin{proof} 
        Since $|q|\geq\max\{35, \mathfrak{q}(\mathbb{S}^3_L)\}$, Lemma \ref{lemma:numerical} with $L'=L$ tells us that the core curve $c$ of $\mathbb{S}^3_L(p/q)$ is a geodesic of length $l(c) < \min\{0.0706, {1.9793}^{-1}\cdot\sys(\mathbb{S}^3_L)\}$. Moreover, any other geodesic $\gamma$ in $\mathbb{S}^3_L(p/q)$ which is shorter than $c$, or even just another geodesic with length $l(\gamma) < \min\{0.0706, {1.9793}^{-1}\cdot\sys(\mathbb{S}^3_L)\}$, must also be isotopic to a geodesic in $\mathbb{S}^3_L(p/q)\setminus\nu(c)\cong\mathbb{S}^3_L$, where its length is strictly less than $\min\{0.14, \sys(\mathbb{S}^3_L)\}$. However, all geodesics in $\mathbb{S}^3_L$ clearly have length at least $\sys(\mathbb{S}^3_L)$. Therefore the shortest geodesic in $\mathbb{S}^3_L(p/q)$ is precisely the core curve $c$. 
        
        Since $|\tilde{q}|\geq\max\{35, \mathfrak{q}(\mathbb{S}^3_L)\}$, Lemma \ref{lemma:numerical} applied to $\tilde{L}$ with $L'=L$ tells us that the core curve $\tilde{c}$ of $\mathbb{S}^3_{\tilde{L}}(\tilde{p}/\tilde{q})$ is a geodesic of length $l(\tilde{c}) < \min\{0.0706, {1.9793}^{-1}\cdot\sys(\mathbb{S}^3_L)\}$. By Mostow rigidity, the slope-preserving homeomorphism $\mathbb{S}^3_L(p/q) \cong \mathbb{S}^3_{\tilde{L}}(\tilde{p}/\tilde{q})$ is homotopic to an isometry, which preserves geodesics and their lengths. This means that the shortest geodesic in $\mathbb{S}^3_{\tilde{L}}(\tilde{p}/\tilde{q})$ must in fact be the core curve $\tilde{c}$, as any shorter geodesic $\tilde{\gamma}$ would have to correspond to another geodesic $\gamma$ in $\mathbb{S}^3_L(p/q)$ with $l(\gamma) < \min\{0.0706, {1.9793}^{-1}\cdot\sys(\mathbb{S}^3_L)\}$. Since the isometry preserves shortest geodesics, the core curves $c$ in $\mathbb{S}^3_L(p/q)$ and $\tilde{c}$ in $\mathbb{S}^3_{\tilde{L}}(\tilde{p}/\tilde{q})$ are identified. This gives the desired slope-preserving homeomorphism taking $L_0$ to $\tilde{L}_0$; it follows that $L=\tilde{L}$.
    \end{proof} 
\end{proposition}

\subsection{Proof of Theorem \ref{theorem:new}} 

Set $L_0=K$ to be our fixed hyperbolic knot and $\tilde{L}_0=\tilde{K}$ to be another hyperbolic knot for which there is an orientation-preserving homeomorphism $\mathbb{S}^3_K(p/q) \cong \mathbb{S}^3_{\tilde{K}}(\tilde{p}/\tilde{q})$. We will see that the assumption $|q|, |\tilde{q}| \geq \max\{35, \mathfrak{q}(\mathbb{S}^3_K)\}$ ensures that the shortest geodesics in both fillings are exactly the core curves, $c_K$ and $c_{\tilde{K}}$. This enables us to complete the proof of Theorem \ref{theorem:new}. 

\begin{proof}[Proof of Theorem \ref{theorem:new}] 
    Suppose that $\mathbb{S}^3_K(p/q)\cong\mathbb{S}^3_{K'}(p/q)$ for a hyperbolic knot $K$ and slope $p/q$ satisfying the hypotheses in the statement of the theorem. Recall from Proposition \ref{proposition:knot} that $K'$ is either (i) another hyperbolic knot or (ii) a cable $K'=C_{r,s}(\hat{K})$ of one. In case (ii), taking \mbox{$(\tilde{K},\tilde{p}/\tilde{q})=(\hat{K},p/qs^2)$} in Proposition \ref{proposition:35} tells us that $K'=C_{r,s}(K)$ and \mbox{$\mathbb{S}^3_K(p/q)\cong\mathbb{S}^3_{K'}(p/q)\cong\mathbb{S}^3_K(p/qs^2)$}, which can be ruled out by Corollary \ref{corollary:csc}. This leaves us with case (i) as the only possibility; taking $(\tilde{K},\tilde{p}/\tilde{q})=(K',p/q)$ in Proposition \ref{proposition:35} then gives $K=K'$, as required. 
\end{proof} 

\begin{proof}[Proof of Corollary \ref{corollary:new}] 
    For the $t$-twist knots $K=T^\pm_t$, the SnapPy length spectrum function returns the length $\Lambda^\pm_t$ of a shortest geodesic in the complement. These are listed in Table \ref{table:twist_lengths} for $|t|\leq34$. These values are clearly all greater than $0.14$ for $|t|\leq3$ (provided $T^{\pm}_t$ is hyperbolic) but less than $0.14$ for $4\leq|t|\leq34$ by inspection. 
    
    For $|t|\geq35$, we claim that the length of the shortest geodesic always satisfies $\Lambda^\pm_t < 0.14$ (and thus $\mathfrak{q}(\mathbb{S}^3_{T^\pm_t}) \geq 35$). Observe that $\mathbb{S}^3_{W^\pm}(1/t) \cong \mathbb{S}^3_{T^\pm_t}$ and note that the Whitehead link complement satisfies $$\sys(\mathbb{S}^3_{W^\pm}) \approx 1.061275061905 > 0.14.$$ By Lemma \ref{lemma:numerical}, if $|t|\geq \max\{35, \mathfrak{q}(\mathbb{S}^3_{W^\pm})\} = 35$, then the core curve $c_{V^\pm}$ of $\mathbb{S}^3_{W^\pm}(1/t)$ is a geodesic whose length satisfies $l(c_{V^\pm}) < \min\{0.0706, 1.9793^{-1} \cdot \sys(\mathbb{S}^3_{W^\pm})\} = 0.0706 < 0.14$, so $\Lambda^\pm_t < 0.14$.  
\end{proof} 

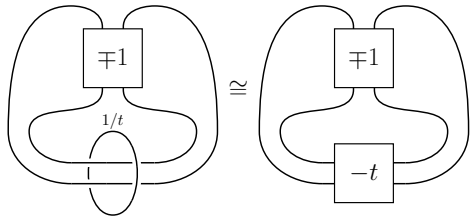
\begin{figure}[htbp!] 
    \centering 
    \resizebox{0.5\textwidth}{!}{\centering
\begin{tikzpicture} [squarednode/.style={rectangle, draw=black, fill=white, thick, minimum size=40pt}]
\begin{scope}[xshift=-3cm]
    \begin{knot}[end tolerance=1pt] 
    \flipcrossings{1,3}
        \strand[white, double=black, thick, double distance=1pt] (-1,-1.25) 
        to [out=left, in=down] (-2.5,0) 
        to [out=up, in=left] (-1,3)
        to [out=right, in=up] (-0.25,2.5)
        to [out=down, in=up] (-0.25,1) 
        to [out=down, in=up] (-2,0)
        to [out=down, in=left] (-1,-0.75); 
        \strand[white, double=black, thick, double distance=1pt] (-1,-0.75) to [out=right, in=left] (1,-0.75); 
        \strand[white, double=black, thick, double distance=1pt] (1,-0.75) 
        to [out=right, in=down] (2,0)
        to [out=up, in=down] (0.25,1)
        to [out=up, in=down] (0.25,2.5) 
        to [out=up, in=left] (1,3) 
        to [out=right, in=up] (2.5,0) 
        to [out=down, in=right] (1,-1.25); 
        \strand[white, double=black, thick, double distance=1pt] (1,-1.25) to [out=left, in=right] (-1,-1.25); 
        \strand[white, double=black, thick, double distance=1pt] (0,0) 
        to [out=right, in=right] (0,-2) 
        to [out=left, in=left] (0,0); 
        \node[] at (0,0.25) {\large $1/t$}; 
    \end{knot} 
    \node[squarednode, fill=white, opacity=1, draw=black] at (0,1.75) {\Large $\mp1$}; 
\end{scope} 
\node[] at (0,1) {\Large $\cong$}; 
\begin{scope}[xshift=3cm]
    \begin{knot}[end tolerance=1pt]
        \strand[white, double=black, thick, double distance=1pt] (-1,-1.25) 
        to [out=left, in=down] (-2.5,0) 
        to [out=up, in=left] (-1,3)
        to [out=right, in=up] (-0.25,2.5)
        to [out=down, in=up] (-0.25,1) 
        to [out=down, in=up] (-2,0)
        to [out=down, in=left] (-1,-0.75); 
        \strand[white, double=black, thick, double distance=1pt] (-1,-0.75) to [out=right, in=left] (1,-0.75); 
        \strand[white, double=black, thick, double distance=1pt] (1,-0.75) 
        to [out=right, in=down] (2,0)
        to [out=up, in=down] (0.25,1)
        to [out=up, in=down] (0.25,2.5) 
        to [out=up, in=left] (1,3) 
        to [out=right, in=up] (2.5,0) 
        to [out=down, in=right] (1,-1.25); 
        \strand[white, double=black, thick, double distance=1pt] (1,-1.25) to [out=left, in=right] (-1,-1.25); 
    \end{knot} 
    \node[squarednode, fill=white, opacity=1, draw=black] at (0,-1) {\Large $-t$}; 
    \node[squarednode, fill=white, opacity=1, draw=black] at (0,1.75) {\Large $\mp1$}; 
\end{scope}
\end{tikzpicture} } 
    \caption{Rolfsen $-t$-twist giving $\mathbb{S}^3_{W^\pm}(1/t) \cong \mathbb{S}^3_{T^\pm_t}$. } 
    \label{figure:1/t} 
\end{figure}

\subsection{Proof of Theorem \ref{theorem:main}} 

Now we will set $L=P$ and $\tilde{L}=\tilde{P}$ to be a pair of hyperbolic patterns, $P = Q \cup U$ and $\tilde{P} = \tilde{Q} \cup \tilde{U}$, for which there is a slope-preserving homeomorphism $\mathbb{S}^3_P(p/q) \cong \mathbb{S}^3_{\tilde{P}}(p/q)$. The assumption $|q|,|\tilde{q}| \geq \max\{35, \mathfrak{q}(\mathbb{S}^3_P)\}$ will imply that the shortest geodesics in both fillings are the core curves, $c_Q$ and $c_{\tilde{Q}}$. Putting everything together, we can complete the proof of Theorem \ref{theorem:main}. 
    
\begin{proof}[Proof of Theorem \ref{theorem:main}]
    Suppose that $\mathbb{S}^3_K(p/q)\cong\mathbb{S}^3_{K'}(p/q)$ for a satellite knot $K=P(J)$ and slope $p/q$ satisfying the hypotheses in the statement of the theorem. Recall from Proposition \ref{proposition:pattern} that, since we have $\gcd(p, w)\neq1$ and $|q|\geq3$, we know that $K'=P'(J)$ is also a satellite knot, by either (i) a hyperbolic pattern $P'$ or (ii) a cable $P'=C_{r,s}(\hat{P})$ of one. In case (ii), taking ($\tilde{P},\tilde{p}/\tilde{q})=(\hat{P},p/qs^2)$ in Proposition \ref{proposition:35} tells us that $K'=C_{r,s}(P(J))=C_{r,s}(K)$ and $\mathbb{S}^3_K(p/q)\cong\mathbb{S}^3_{K'}(p/q)\cong\mathbb{S}^3_K(p/qs^2)$, which can be ruled out by Corollary \ref{corollary:csc}. This leaves us with case (i) as the only possibility; taking $(\tilde{P},\tilde{p}/\tilde{q})=(P',p/q)$ in Proposition \ref{proposition:35} then gives $K=K'$. 
\end{proof} 

\begin{proof}[Proof of Corollary \ref{corollary:main}] 
    For the $n$-clasped $t$-twisted Whitehead links $P = W^n_t$, the SnapPy length spectrum function returns the length $\Lambda^{(n)}$ of a shortest geodesic in $\mathbb{S}^3_{W^n_t}$ (this is independent of $t$). These are listed for $|n|\leq34$ in Table \ref{table:double_lengths} and are clearly all greater than $0.14$ for $|n|\leq4$ but less than $0.14$ for $5\leq|n|\leq34$ by inspection. 
    
    For $|n|\geq35$, we need to show that the length of the shortest geodesic always satisfies $\Lambda^{(n)} < 0.14$. Letting $B$ denote the Borromean rings, observe that we have $\mathbb{S}^3_B(1/n) \cong \mathbb{S}^3_{W^n}$ and note that $\mathbb{S}^3_B$ satisfies $$\sys(\mathbb{S}^3_B) \approx 2.122550123810 > 0.14.$$ By Lemma \ref{lemma:numerical}, if $|n|\geq \max\{35, \mathfrak{q}(\mathbb{S}^3_B)\} = 35$, then the core curve $c$ of $\mathbb{S}^3_{B}(1/n)$ is a geodesic whose length satisfies $l(c) < \min\{0.0706, 1.9793^{-1} \cdot \sys(\mathbb{S}^3_B)\} = 0.0706 < 0.14$. Hence $\Lambda^{(n)} < 0.14$.  
\end{proof} 

\begin{figure}[htbp!] 
    \centering 
    \resizebox{0.5\textwidth}{!}{\centering
\begin{tikzpicture} [squarednode/.style={rectangle, draw=black, fill=white, thick, minimum size=40pt}]
\begin{scope}[xshift=-3cm]
    \begin{knot}[end tolerance=1pt] 
    \flipcrossings{2,3,5,7}
        \strand[white, double=black, thick, double distance=1pt] (-1,-1.25) 
        to [out=left, in=down] (-2.5,0) 
        to [out=up, in=left] (-1,3)
        to [out=right, in=up] (-0.25,2.5) 
        to [out=down, in=up] (-0.25,1); 
        \strand[white, double=black, thick, double distance=1pt] (-0.25,1) 
        to [out=down, in=up] (-2,0)
        to [out=down, in=left] (-1,-0.75); 
        \strand[white, double=black, thick, double distance=1pt] (-1,-0.75) to [out=right, in=left] (1,-0.75); 
        \strand[white, double=black, thick, double distance=1pt] (1,-0.75) 
        to [out=right, in=down] (2,0)
        to [out=up, in=down] (0.25,1); 
        \strand[white, double=black, thick, double distance=1pt] (0.25,1) 
        to [out=up, in=down] (0.25,2.5) 
        to [out=up, in=left] (1,3) 
        to [out=right, in=up] (2.5,0) 
        to [out=down, in=right] (1,-1.25); 
        \strand[white, double=black, thick, double distance=1pt] (1,-1.25) to [out=left, in=right] (-1,-1.25); 
        \strand[white, double=black, thick, double distance=1pt] (0,0) 
        to [out=right, in=right] (0,-2) 
        to [out=left, in=left] (0,0); 
        \strand[white, double=black, thick, double distance=1pt] (-1,2) 
        to [out=down, in=down] (1,2) 
        to [out=up, in=up] (-1,2); 
        \node[] at (0,3) {\large $1/n$}; 
    \end{knot} 
\end{scope} 
\node[] at (0,1) {\Large $\cong$}; 
\begin{scope}[xshift=3cm]
    \begin{knot}[end tolerance=1pt] 
    \flipcrossings{1,3} 
        \strand[white, double=black, thick, double distance=1pt] (-1,-1.25) 
        to [out=left, in=down] (-2.5,0) 
        to [out=up, in=left] (-1,3)
        to [out=right, in=up] (-0.25,2.5)
        to [out=down, in=up] (-0.25,1) 
        to [out=down, in=up] (-2,0)
        to [out=down, in=left] (-1,-0.75); 
        \strand[white, double=black, thick, double distance=1pt] (-1,-0.75) to [out=right, in=left] (1,-0.75); 
        \strand[white, double=black, thick, double distance=1pt] (1,-0.75) 
        to [out=right, in=down] (2,0)
        to [out=up, in=down] (0.25,1)
        to [out=up, in=down] (0.25,2.5) 
        to [out=up, in=left] (1,3) 
        to [out=right, in=up] (2.5,0) 
        to [out=down, in=right] (1,-1.25); 
        \strand[white, double=black, thick, double distance=1pt] (1,-1.25) to [out=left, in=right] (-1,-1.25); 
        \strand[white, double=black, thick, double distance=1pt] (0,0) 
        to [out=right, in=right] (0,-2) 
        to [out=left, in=left] (0,0); 
    \end{knot} 
    \node[squarednode, fill=white, opacity=1, draw=black] (B) at (0,1.75) {\Large $-n$}; 
\end{scope}
\end{tikzpicture} } 
    \caption{Rolfsen $-n$-twist giving $\mathbb{S}^3_B(1/n) \cong \mathbb{S}^3_{W^n}$. }
    \label{figure:1/n} 
\end{figure}
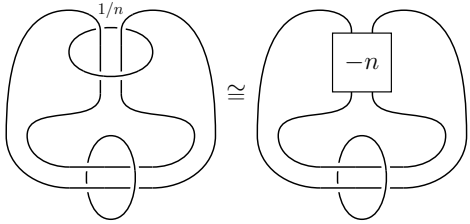

\section{Minimal volumes} 
\label{section:volume} 

We will now consider the special case where our fixed hyperbolic pattern is the $\pm1$-clasped $t$-twisted Whitehead link, $W^\pm_t = V^\pm_t \cup U$, for $\pm t \in \{-1,0,1,2\}$ as in Theorem \ref{theorem:refined}. To clarify this notation: we take $t \in \{-1,0,1,2\}$ in the positively clasped case and $t \in \{1,0,-1,-2\}$ in the negatively clasped case (just as we did for the selection of twist knots $T^\pm_t$ depicted in Figure \ref{figure:twist}). 

The Whitehead link has many significant properties. We know that $W^\pm_t$ has linking number $w=0$. The untwisted version, $W^\pm = V^\pm \cup U$, comprises two unknotted components which can be exchanged by an isotopy. Most importantly, the Whitehead link complement (jointly with its sister manifold, the $P(-2,3,8)$ pretzel link complement) achieves the least hyperbolic volume over all 2-cusped orientable hyperbolic 3-manifolds \cite{agol}. This leads to an alternative method of deducing that the existence of a slope-preserving homeomorphism $\mathbb{S}^3_{W^\pm_t}(p/q) \cong \mathbb{S}^3_{\tilde{P}}(\tilde{p}/\tilde{q})$ (where $\tilde{P}$ is a hyperbolic pattern) implies that $\tilde{P}=W^\pm_t$. This will form a crucial step in our upcoming proof of Theorem \ref{theorem:refined}. 

\begin{theorem:refined}
    Let $K=W^\pm_t(J)$ be a $\pm1$-clasped $t$-twisted Whitehead double for \mbox{$\pm t\in \{-1, 0, 1, 2\}$}. 
    
    Then every slope $p/q$ with $|p|\neq1$ and $|q|\geq q_{\min}$ is characterising for $K$, where $q_{\min}$ is some constant determined by the Whitehead link and independent of $J$. 
    
    Furthermore, if the SnapPy census of 2-cusped orientable hyperbolic 3-manifolds is complete up to stage $k$ of Table \ref{table:stages} (namely, up to the volume bound $V_k$, or up to the first $a_k$ items), then we can take $q_{\min}$ to be $q_k$. 
\end{theorem:refined} 

Our key ingredient will be an inequality that restricts the extent to which Dehn filling can decrease hyperbolic volume. This depends on a scaling factor which is determined by the length (and in turn by the denominator) of the filling slope. That is, for a given lower bound $q_{\min}$ on $|q|$, we can construct an upper bound $V_{\max}$ on the volume of $\mathbb{S}^3_{\tilde{P}}$ for which it is possible to have a slope-preserving homeomorphism $\mathbb{S}^3_{W^\pm_t}(p/q) \cong \mathbb{S}^3_{\tilde{P}}(\tilde{p}/\tilde{q})$. Under the additional requirement of completeness of the SnapPy census of 2-cusped orientable hyperbolic 3-manifolds up to the given $V_{\max}$, we can use the fact that $\tilde{P} = \tilde{Q} \cup \tilde{U}$ must have winding number $\tilde{w}=0$ to systematically rule out everything except $\mathbb{S}^3_{\tilde{P}}\cong\mathbb{S}^3_{W^\pm_t}$. Finally, since $\pm t\in\{-1,0,1,2\}$, we can deduce that $\tilde{P} = W^\pm_t$.

\subsection{Homeomorphic link complements}

Unlike knots, links are not necessarily determined by their complements. However, in the case of the Whitehead link, there are a limited number of possibilities. 

\begin{lemma} 
\label{lemma:twisted}
    Let $W^\pm_t = V^\pm_t \cup U$ be a $\pm1$-clasped $t$-twisted Whitehead link for some $t\in\mathbb{Z}$ and let $\tilde{P} = \tilde{Q} \cup \tilde{U}$ be a link such that there exists an orientation-preserving homeomorphism $\mathbb{S}^3_{W^\pm_t} \cong \mathbb{S}^3_{\tilde{P}}$. 

    Then $\tilde{P}$ is a $t'$-twisted Whitehead link with the same clasp, $\tilde{P}=W^\pm_{t'}$, for some integer $t' \in\mathbb{Z}$. 
    \begin{proof} 
        Since $W^\pm_t$ is Brunnian with linking number zero, $\tilde{P}$ must be obtained from $W^\pm_t$ by an integer number of full twists around an unknotted component \cite{brunnian}. Each clasp sign generates a distinct set of twisted Whitehead links with orientation-preserving homeomorphic complements (where objects in opposing sets are related by orientation-reversing homeomorphisms). 
    \end{proof}
\end{lemma} 

We will later use the following result to analyse cables of twisted Whitehead doubles. 

\begin{lemma} 
\label{lemma:cabled_twisted}
    Let $W^\pm_t = V^\pm_t \cup U$ be a $\pm1$-clasped $t$-twisted Whitehead link for $\pm t\in\{-1,0,1,2\}$. If there exists a slope-preserving homeomorphism $\mathbb{S}^3_{W^\pm_t}(p/q)\cong\mathbb{S}^3_{P'}(p/q)$, where $P'=C_{r,s}(W^\pm_{t'})$ is a cable of some $W^\pm_{t'} = V^\pm_{t'} \cup U'$ and $p/q$ is a non-integral slope with $|p-qrs|=1$, then $t=t'=0$ and $|r|=1$. 
    \begin{proof}
        Since $V^\pm_t \in \{U,\pm T,S,\pm R\}$, we can apply Lemma \ref{lemma:knotting} to rule out $S$ and $\pm R$. In the case where $V^\pm_t=U$, we have $t=t'=0$ and $|r|=1$, as required. Alternatively, if $V^\pm_t=\pm T$, then we have $t=\pm1$, $t'=0$ and $(r,s)\in\{(\pm2,3),(\pm3,2)\}$. However, by the Rolfsen $\mp1$-twist trick shown in Figure \ref{figure:rolfsen}, this gives $$L(p,q) \cong \mathbb{S}^3_{T^\pm_0}(p/q) \cong \mathbb{S}^3_{V^\pm_{\pm1},U}(p/q,\mp1) \cong \mathbb{S}^3_{V^\pm_0,U}(p/qs^2,\mp1) \cong \mathbb{S}^3_{T^\pm_{\mp1}}(p/qs^2),$$ 
        yet the figure eight knot $T^\pm_{\mp1}=S$ has no lens space surgeries. 
    \end{proof} 
\end{lemma} 

\begin{figure}[htbp!]
    \centering 
    \resizebox{\textwidth}{!}{\centering
\begin{tikzpicture} [squarednode/.style={rectangle, draw=black, fill=white, thick, minimum size=40pt}] 
\begin{scope}[xshift=-9cm]
    \begin{knot}[end tolerance=1pt]  
    \flipcrossings{1}
        \strand[white, double=black, thick, double distance=1pt] (-1,-1.25) 
        to [out=left, in=down] (-2.5,0) 
        to [out=up, in=left] (-1,3)
        to [out=right, in=up] (-0.25,2.5)
        to [out=down, in=up] (-0.25,1) 
        to [out=down, in=up] (-2,0)
        to [out=down, in=left] (-1,-0.75); 
        \strand[white, double=black, thick, double distance=1pt] (-1,-0.75) to [out=right, in=left] (1,-0.75); 
        \strand[white, double=black, thick, double distance=1pt] (1,-0.75) 
        to [out=right, in=down] (2,0)
        to [out=up, in=down] (0.25,1)
        to [out=up, in=down] (0.25,2.5) 
        to [out=up, in=left] (1,3) 
        to [out=right, in=up] (2.5,0) 
        to [out=down, in=right] (1,-1.25); 
        \strand[white, double=black, thick, double distance=1pt] (1,-1.25) to [out=left, in=right] (-1,-1.25); 
        \node[] at (2.5,3) {\large $p/q$}; 
        \node[] at (0,-2.5) {\Large $T^{\pm}_{0}=U$}; 
    \end{knot} 
    \node[squarednode, fill=white, opacity=1, draw=black] at (0,1.75) {\Large $\mp1$}; 
\end{scope}
\node[] at (-6,1) {\Large $\cong$};
\begin{scope}[xshift=-3cm]
    \begin{knot}[end tolerance=1pt] 
    \flipcrossings{1,4} 
        \strand[white, double=black, thick, double distance=1pt] (-1,-1.25) 
        to [out=left, in=down] (-2.5,0) 
        to [out=up, in=left] (-1,3)
        to [out=right, in=up] (-0.25,2.5)
        to [out=down, in=up] (-0.25,1) 
        to [out=down, in=up] (-2,0)
        to [out=down, in=left] (-1,-0.75); 
        \strand[white, double=black, thick, double distance=1pt] (-1,-0.75) to [out=right, in=left] (1,-0.75); 
        \strand[white, double=black, thick, double distance=1pt] (1,-0.75) 
        to [out=right, in=down] (2,0)
        to [out=up, in=down] (0.25,1)
        to [out=up, in=down] (0.25,2.5) 
        to [out=up, in=left] (1,3) 
        to [out=right, in=up] (2.5,0) 
        to [out=down, in=right] (1,-1.25); 
        \strand[white, double=black, thick, double distance=1pt] (1,-1.25) to [out=left, in=right] (-1,-1.25); 
        \strand[white, double=black, thick, double distance=1pt] (-3.25,0) 
        to [out=down, in=down] (-1.25,0) 
        to [out=up, in=up] (-3.25,0); 
        \node[] at (-3,-1) {\large $\mp1$}; 
        \node[] at (2.5,3) {\large $p/q$}; 
        \node[] at (0,-2.5) {\Large $T^{\pm1}_{\pm1} = \pm T$}; 
    \end{knot} 
    \node[squarednode, fill=white, opacity=1, draw=black] at (0,-1) {\Large $\mp1$}; 
    \node[squarednode, fill=white, opacity=1, draw=black] at (0,1.75) {\Large $\mp1$}; 
\end{scope}
\node[] at (0,1) {\Large $\cong$};
\begin{scope}[xshift=3cm]
    \begin{knot}[end tolerance=1pt] 
    \flipcrossings{1,4} 
        \strand[white, double=black, thick, double distance=1pt] (-1,-1.25) 
        to [out=left, in=down] (-2.5,0) 
        to [out=up, in=left] (-1,3)
        to [out=right, in=up] (-0.25,2.5)
        to [out=down, in=up] (-0.25,1) 
        to [out=down, in=up] (-2,0)
        to [out=down, in=left] (-1,-0.75); 
        \strand[white, double=black, thick, double distance=1pt] (-1,-0.75) to [out=right, in=left] (1,-0.75); 
        \strand[white, double=black, thick, double distance=1pt] (1,-0.75) 
        to [out=right, in=down] (2,0)
        to [out=up, in=down] (0.25,1)
        to [out=up, in=down] (0.25,2.5) 
        to [out=up, in=left] (1,3) 
        to [out=right, in=up] (2.5,0) 
        to [out=down, in=right] (1,-1.25); 
        \strand[white, double=black, thick, double distance=1pt] (1,-1.25) to [out=left, in=right] (-1,-1.25); 
        \strand[white, double=black, thick, double distance=1pt] (-3.25,0) 
        to [out=down, in=down] (-1.25,0) 
        to [out=up, in=up] (-3.25,0); 
        \node[] at (-3,-1) {\large $\mp1$}; 
        \node[] at (2.5,3) {\large $p/qs^2$}; 
        \node[] at (0,-2.5) {\Large $T^{\pm1}_{0}=U$};
    \end{knot} 
    \node[squarednode, fill=white, opacity=1, draw=black] at (0,1.75) {\Large $\mp1$}; 
\end{scope}
\node[] at (6,1) {\Large $\cong$};
\begin{scope}[xshift=9cm]
    \begin{knot}[end tolerance=1pt] 
        \strand[white, double=black, thick, double distance=1pt] (-1,-1.25) 
        to [out=left, in=down] (-2.5,0) 
        to [out=up, in=left] (-1,3)
        to [out=right, in=up] (-0.25,2.5)
        to [out=down, in=up] (-0.25,1) 
        to [out=down, in=up] (-2,0)
        to [out=down, in=left] (-1,-0.75); 
        \strand[white, double=black, thick, double distance=1pt] (-1,-0.75) to [out=right, in=left] (1,-0.75); 
        \strand[white, double=black, thick, double distance=1pt] (1,-0.75) 
        to [out=right, in=down] (2,0)
        to [out=up, in=down] (0.25,1)
        to [out=up, in=down] (0.25,2.5) 
        to [out=up, in=left] (1,3) 
        to [out=right, in=up] (2.5,0) 
        to [out=down, in=right] (1,-1.25); 
        \strand[white, double=black, thick, double distance=1pt] (1,-1.25) to [out=left, in=right] (-1,-1.25); 
        \node[] at (2.5,3) {\large $p/qs^2$}; 
        \node[] at (0,-2.5) {\Large $T^{\pm1}_{\mp1}=S$};
    \end{knot} 
    \node[squarednode, fill=white, opacity=1, draw=black] at (0,-1) {\Large $\pm1$}; 
    \node[squarednode, fill=white, opacity=1, draw=black] at (0,1.75) {\Large $\mp1$}; 
\end{scope}
\end{tikzpicture} } 
    \caption{Rolfsen $\mp1$-twist trick.}
\label{figure:rolfsen}
\end{figure}
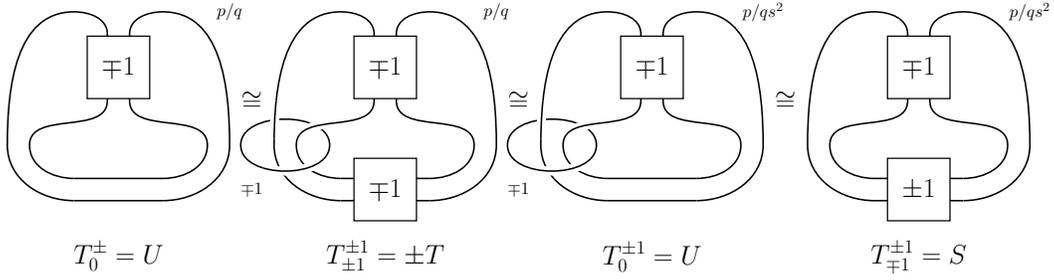

In general, we can use linking number as an invariant to distinguish between link complements. 

\begin{lemma} 
\label{lemma:linking}
    Let $P = Q \cup U, \tilde{P} = \tilde{Q} \cup \tilde{U}$ be links whose linking numbers $w,\tilde{w}$ satisfy $|w|\neq |\tilde{w}|$. 
    
    Then there exists no orientation-preserving homeomorphism $\mathbb{S}^3_P \cong \mathbb{S}^3_{\tilde{P}}$ between their complements. 
    \begin{proof}
        Consider the map $H_1(\mathbb{T}^2) \to H_1(\mathbb{S}^3_P)$ induced by the inclusion $\mathbb{T}^2 \hookrightarrow \mathbb{S}^3_P$ of either toroidal boundary component. The index of the image of this map is 
        $$
        \begin{cases}
            |w| & \text{if} \ w\neq0, \\
            \infty & \text{if} \ w=0. \\
        \end{cases}
        $$
        Under any homeomorphism $\mathbb{S}^3_P \cong \mathbb{S}^3_{\tilde{P}}$, this index would be preserved, giving $|w|=|\tilde{w}|$. 
    \end{proof}
\end{lemma} 

Recall that the Whitehead link $W^\pm_t$ has linking number $w=0$. We already know from Lemma \ref{lemma:winding} that, in order to have our slope-preserving homeomorphism $\mathbb{S}^3_{W^\pm_t}(p/q)\cong\mathbb{S}^3_{\tilde{P}}(\tilde{p}/\tilde{q})$, $\tilde{P}$ should also have linking number $\tilde{w}=0$. As a consequence, we can obstruct candidates for $\tilde{P}$ by showing that for each 2-cusped orientable hyperbolic 3-manifold $M$ with $\vol(M)< V_{\max}$, if there exists an orientation-preserving homeomorphism $M\cong\mathbb{S}^3_{\tilde{P}}$ to the complement of any hyperbolic link $\tilde{P} = \tilde{Q} \cup \tilde{U}$, then $\tilde{P}$ must have linking number $\tilde{w}\neq0$. This rules out $M$ as a potential pattern piece.

\subsection{Volume inequalities} 

Our main technique in addressing the case of Whitehead doubles will be the use of hyperbolic volume. In particular, we rely on the following inequalities to construct our upper bound $V_{\max}$ on $\vol(\mathbb{S}^3_{\tilde{P}})$ for any link complement $\mathbb{S}^3_{\tilde{P}}$ such that $\mathbb{S}^3_{W^\pm_t}(p/q) \cong \mathbb{S}^3_{\tilde{P}}(\tilde{p}/\tilde{q})$. 

\begin{theorem}[\cite{fkp, thurston}]
\label{theorem:inequalities}
    Let $M$ be a compact orientable hyperbolic 3-manifold with toroidal boundary $\partial M = \sqcup_{i=1}^m \mathbb{T}^2_i$. Take a subset $\{\mathbb{T}^2_{i_j}\}_{j=1}^n$ of boundary tori together with slopes $\{\gamma_{i_j} \subset \mathbb{T}^2_{i_j}\}_{j=1}^n$. Choose horoball neighbourhoods $N_{i_j}$ of each $\mathbb{T}^2_{i_j}$ such that each $l(\gamma_{i_j}) \geq l_{\min} > 2\pi$.
    
    Then $M(\gamma_{i_1}, \ldots, \gamma_{i_n})$ is a hyperbolic 3-manifold such that $$\Big(1-\Big(\frac{2\pi}{l_{\min}}\Big)^2\Big)^{3/2} \cdot \vol(M) \leq \vol(M(\gamma_{i_1}, \ldots, \gamma_{i_n})) < \vol(M).$$
\end{theorem}

In our case, $M$ will be the complement of a 2-component hyperbolic link (so $m=2$) and we will fill one boundary component (so $n=1$) along slope $\tilde{p}/\tilde{q}$. 

\begin{lemma}
\label{lemma:bound}
     Let $W^\pm_t = V^\pm_t \cup U$ be a $\pm1$-clasped $t$-twisted Whitehead link for some integer $t\in\mathbb{Z}$ and let $\tilde{P} = \tilde{Q} \cup \tilde{U}$ be a hyperbolic pattern. Suppose that there exists an orientation-preserving homeomorphism $\mathbb{S}^3_{W^\pm_t}(p/q) \cong \mathbb{S}^3_{\tilde{P}}(\tilde{p}/\tilde{q})$ for some slopes $p/q, \tilde{p}/\tilde{q}$ on the boundary components of $\mathbb{S}^3_{W^\pm_t}, \mathbb{S}^3_{\tilde{P}}$ corresponding to $V^\pm_t, \tilde{Q}$, respectively, with $|\tilde{q}|\geq q_{\min}$ for some fixed $q_{\min}\geq11$. 
    
    Then $q_{\min}$ determines a constant $$V_{\max} := \Big(1-3\Big(\frac{2\pi}{q_{\min}}\Big)^2\Big)^{-3/2} \cdot \vol(\mathbb{S}^3_{W^\pm_t})$$ such that $\vol(\mathbb{S}^3_{\tilde{P}}) <V_{\max}$. 
    \begin{proof}
        Since $|\tilde{q}|\geq q_{\min}\geq11$, the length of the slope $\tilde{p}/\tilde{q}$ satisfies 
        $$l(\tilde{p}/\tilde{q}) \geq \frac{2\sqrt{3}}{6} \cdot \Delta(\tilde{p}/\tilde{q},1/0) = \frac{|\tilde{q}|}{\sqrt{3}} \geq \frac{q_{\min}}{\sqrt{3}} =: l_{\min} > 2\pi$$ 
        by Lemma \ref{lemma:facts} and Theorem \ref{theorem:6} (due to the existence of the exceptional filling $\mathbb{S}^3_{\tilde{P}}(1/0)\cong\mathbb{S}^1\times\mathbb{D}^2$). Therefore we can apply Theorem \ref{theorem:inequalities} to obtain the inequality 
        $$\Big(1-3\Big(\frac{2\pi}{q_{\min}}\Big)^2\Big)^{3/2} \cdot \vol(\mathbb{S}^3_{\tilde{P}}) \leq \vol(\mathbb{S}^3_{\tilde{P}}(\tilde{p}/\tilde{q})) = \vol(\mathbb{S}^3_{W^\pm_t}(p/q)) < \vol(\mathbb{S}^3_{W^\pm_t})$$  
        and rearrange to reach the desired result.  
    \end{proof}
\end{lemma}

Note that we require $q_{\min} \geq 11$ in order for the hypothesis $l_{\min}>2\pi$ to be satisfied. We also need to assume that $q_{\min} \geq 24$ in order to ensure that we obtain a finite list from SnapPy; otherwise, the volume bound $V_{\max}$ becomes larger than the volume of the 3-chain link complement, which is conjecturally the smallest limit point for the set of volumes of 2-cusped orientable hyperbolic 3-manifolds. 

We know that $\mathbb{S}^3_{\tilde{P}}$ is the complement of a 2-component hyperbolic link $\tilde{P} = \tilde{Q} \cup \tilde{U}$ with linking number $\tilde{w}=0$. These constraints will allow us to eliminate most of the 2-cusped orientable hyperbolic 3-manifolds with volume at most $V_{\max}$ from the list by showing that none of them are of this form. As we will see, one way to do this is by comparing different fillings of these manifolds: for homological reasons, the existence of certain types of fillings obstructs the $\tilde{w}=0$ condition.

\subsection{Comparing Dehn fillings} 

When considering pairs of fillings of a 2-cusped orientable hyperbolic 3-manifold $M$, there are several situations which can arise that prevent $M$ from being of the required form. The simplest case occurs when there are multiple solid torus fillings. 

\begin{lemma} 
\label{lemma:berge-gabai}
   Suppose that $M$ is a 2-cusped orientable hyperbolic 3-manifold which admits distinct solid torus fillings, $M((a_1,b_1),\ast)$, $M((\hat{a}_1,\hat{b}_1),\ast)$, on the same boundary component. 

   Then there exists no hyperbolic link $\tilde{P} = \tilde{Q} \cup \tilde{U}$ with winding number $\tilde{w}=0$ such that $M \cong \mathbb{S}^3_{\tilde{P}}$. 
    \begin{proof} 
        We can think of $M((\hat{a}_1,\hat{b}_1),\ast)$ as a solid torus surgery on the core of the filling in the solid torus $M((a_1,b_1),\ast)$. Since $M$ is hyperbolic, and thus irreducible, the core cannot be contained in any $\mathbb{B}^3 \subset \mathbb{S}^1 \times \mathbb{D}^2$. We can now apply the classification of such solid torus fillings given in \cite{bg} to deduce that the core must be a Berge-Gabai knot: a 0-bridge or 1-bridge braid in $\mathbb{S}^1 \times \mathbb{D}^2$, which clearly has non-zero winding number in both cases. Moreover, if $M \cong \mathbb{S}^3_{\tilde{P}}$, then this core corresponds to one of the link components sitting inside the complement of the other. It follows that $\tilde{w} \neq 0$. 
    \end{proof} 
\end{lemma} 

If we can only find a single solid torus filling on a component of $M$, then we check for the following situation instead. 

\begin{lemma} 
\label{lemma:free} 
    Suppose that $M$ is a 2-cusped orientable hyperbolic 3-manifold which admits a solid torus filling $M((a_1,b_1), \ast)\cong\mathbb{S}^1\times\mathbb{D}^2$. Let $M((a_1,b_1), (a_2,b_2))$ be the unique $\mathbb{S}^1 \times \mathbb{S}^2$ filling of this $\mathbb{S}^1 \times \mathbb{D}^2$ and suppose that there also exists a filling $M((a_1',b_1'),(a_2,b_2))$ whose first homology is finite. 
    
   Then there exists no hyperbolic link $\tilde{P} = \tilde{Q} \cup \tilde{U}$ with winding number $\tilde{w}=0$ such that $M \cong \mathbb{S}^3_{\tilde{P}}$. 
    \begin{proof}
        Suppose for a contradiction that $M\cong\mathbb{S}^3_{\tilde{P}}$ for such a link $\tilde{P}$. Then we can express the filling data in the form of slopes $(a,b)=\alpha/\beta$ with respect to the meridian and longitude of each link component. Since $\tilde{P}$ has winding number $\tilde{w}=0$, the first homology of a filling along both components can easily be written down. In particular, consider the following two lens space fillings of our $\mathbb{S}^1\times \mathbb{D}^2$: 
        \begin{align*}
            \mathbb{Z} \cong H_1(\mathbb{S}^1 \times \mathbb{S}^2) \cong H_1(M(\alpha_1/\beta_1, \alpha_2/\beta_2)) \cong \mathbb{Z}/\alpha_1\mathbb{Z} \oplus \mathbb{Z}/\alpha_2\mathbb{Z}, \\
            \mathbb{Z}/2\mathbb{Z} \cong H_1(\mathbb{RP}^3) \cong H_1(M(\alpha_1/\beta_1, \alpha_2'/\beta_2')) \cong \mathbb{Z}/\alpha_1\mathbb{Z} \oplus \mathbb{Z}/\alpha_2'\mathbb{Z}. 
        \end{align*}
        It follows that $|\alpha_1|=1$, $|\alpha_2|=0$ and $|\alpha_2'|=2$. But then $H_1(M(\alpha_1'/\beta_1', \alpha_2/\beta_2))\cong\mathbb{Z}/\alpha_1'\mathbb{Z} \oplus\mathbb{Z}$.
    \end{proof}
\end{lemma} 

We will now use these results to prove Theorem \ref{theorem:refined}.

\subsection{Proof of Theorem \ref{theorem:refined}} 

Given a slope-preserving homeomorphism $\mathbb{S}^3_{W^\pm_t}(p/q)\cong\mathbb{S}^3_{\tilde{P}}(\tilde{p}/\tilde{q})$, we need to show that $\tilde{P}=W^\pm_t$. We first claim that $\tilde{P}=W^\pm_{t'}$ for some $t'\in\mathbb{Z}$. 

\begin{table}[htbp!]
\centering 
\begin{multicols}{2} 
    \begin{tabular}{|r|c|c|} \hline 
        $M$ & $H_1(M)$ & $|\operatorname{Tor}(H_1(M))|$ \\ 
        \hline \hline 
        $m412$ & $\mathbb{Z} \oplus \mathbb{Z}/2\mathbb{Z}$ & 2 \\ 
        \hline \hline 
        $s596$ & $\mathbb{Z} \oplus \mathbb{Z}/2\mathbb{Z}$ & 2 \\ 
        \hline 
    \end{tabular} 
\newline 
\vspace{1em} 
\newline 
    \begin{tabular}{|r|l|c|} \hline 
        $M$ & $\tilde{P}$ with $M \cong \mathbb{S}^3_{\tilde{P}}$ & $|\tilde{w}|\neq0$ \\\hline \hline 
        $m125$ & $L13n5885$ & 5 \\ 
        $m203$ & $L6a2$ & 3 \\ 
        $m295$ & $L9n14$ & 1 \\ 
        $m367$ & $L7a6 $& 1 \\ 
        $m391$ & $L8a12$ & 4 \\ 
        \hline \hline 
        $s443$ & $L11n208$ & 2 \\ 
        $s578$ & $L9a364$ & 2 \\ 
        $s602$ & $L10a114$ & 5 \\ 
        $s647$ & $L12n1027$ & 2 \\ 
        \hline \hline 
        $ v1060 $ & $L13n5895$ & 3 \\ 
        $ v1263 $ & $L11a360 $ & 3 \\ 
        \hline
    \end{tabular}    
\newline
\columnbreak
\newline 
\vspace{1em} 
    \begin{tabular}{|r|l|} \hline 
        $M$ & 
        $\mathbb{S}^1\times\mathbb{D}^2$ fillings \\ 
        \hline \hline 
        $ m202 $ & (1,0), (0, 1), (-1, 1) \\ 
        $ m329 $ & (1,0), (0, 1) \\ 
        $ m357 $ & (1,0), (0, 1) \\ 
        $ m366 $ & (1,0), (0, 1) \\ 
        $ m388 $ & (1,0), (0, 1) \\ 
        \hline \hline 
        $ s503 $ & (1,0), (0, 1) \\ 
        $ s548 $ & (1,0), (0, 1) \\ 
        $ s579 $ & (1,0), (0, 1) \\ 
        $ s601 $ & (1,0), (0, 1) \\ 
        \hline \hline 
        $ v1180 $ & (1,0), (0, 1) \\ 
        $ v1203 $ & (1,0), (0, 1) \\ 
        $ v1264 $ & (1,0), (0, 1) \\ 
        \hline \hline 
        $ t02728 $ & (1,0), (0, 1) \\  
        $ t02750 $ & (1,0), (0, 1) \\ 
        \hline 
    \end{tabular}  
\vfill 
\end{multicols} 
\caption{2-cusped orientable hyperbolic 3-manifolds which can be ruled out by one of the simpler elimination techniques (as described in the proof of Theorem \ref{theorem:refined}).} 
\label{table:elimination}
\end{table} 

\begin{proposition} 
\label{proposition:elimination}
    Let $W^\pm_t = V^\pm_t \cup U$ be a $\pm1$-clasped $t$-twisted Whitehead link for some integer $t\in\mathbb{Z}$ and let $\tilde{P} = \tilde{Q} \cup \tilde{U}$ be a hyperbolic pattern. Let $p/q, \tilde{p}/\tilde{q}$ be slopes on the boundary components of $\mathbb{S}^3_{W^\pm_t}, \mathbb{S}^3_{\tilde{P}}$ corresponding to $V^\pm_t, \tilde{Q}$, respectively. Assume that the SnapPy census of 2-cusped orientable hyperbolic 3-manifolds is complete up to stage $k$ in Table \ref{table:stages} (namely, up to the volume bound $V_k$, or up to the first $a_k$ items) and suppose that $|q|,|\tilde{q}|\geq q_k$. 
    
    If there exists a slope-preserving homeomorphism $\mathbb{S}^3_{W^\pm_t}(p/q)\cong\mathbb{S}^3_{\tilde{P}}(\tilde{p}/\tilde{q})$, then $\tilde{P}=W^\pm_{t'}$ for some $t'\in\mathbb{Z}$. 
    \begin{proof}
        Suppose that $\mathbb{S}^3_{W^\pm_t}(p/q)\cong\mathbb{S}^3_{\tilde{P}}(\tilde{p}/\tilde{q})$. Since $|\tilde{q}|\geq q_k\geq11$, Lemma \ref{lemma:bound} produces a constant $$V_k:= \Big(1-3\Big(\frac{2\pi}{q_k}\Big)^2\Big)^{-3/2} \cdot \vol(\mathbb{S}^3_{W^\pm_t})$$ such that $\vol(\mathbb{S}^3_{\tilde{P}}) < V_{k}$. We will show that for any 2-cusped orientable hyperbolic 3-manifold $M$ in the SnapPy census with $\vol(M)<V_k$, a careful analysis shows that $M$ cannot be homeomorphic to the complement of a hyperbolic link $\tilde{P} = \tilde{Q} \cup \tilde{U}$ with winding number $\tilde{w}=0$ unless $\tilde{P} = W^\pm_{t'}$ for some $t' \in \mathbb{Z}$. 
        
        Note that the values of $q_k$ in Table \ref{table:stages} are bounded both above and below: $q_k \leq 43$ ensures that $V_k$ is sufficiently larger than $\vol(\mathbb{S}^3_{W^\pm_t})$ to allow the possibility of other $M$ with $\vol(M)<V_k$; $q_k \geq 24$ implies that $V_k$ is strictly smaller than the volume of the 3-chain link complement, thus ensuring that there are at most finitely many $M$ with $\vol(M) < V_k$ to consider. 
    
        We proceed to work through the SnapPy census of 2-cusped orientable hyperbolic 3-manifolds in the stages described in Table \ref{table:stages}. Namely, for each stage $k$, if the list is complete up to the first $a_k$ items, then $q_k$ is selected to be precisely the least integer to produce a value of $V_k$ which lies just below the highest volume in this list, leaving finitely many manifolds of smaller volume to be systematically eliminated via one of our techniques. 
    
        First, we locate the manifolds which are obviously not link complements for homological reasons. The homology $H_1(\mathbb{S}^3_{\tilde{P}})\cong\mathbb{Z}\oplus\mathbb{Z}$ of a 2-component link complement $\mathbb{S}^3_{\tilde{P}}$ is freely generated by the homology classes of the two meridians, so anything with torsion can be immediately ruled out. It turns out that the only two to consider are $m412$ and $s596$, which both share a volume of $5.0747080321$ and first homology $\mathbb{Z} \oplus \mathbb{Z}/2\mathbb{Z}$. We discount these from all further discussion.  
    
        We would like to use Lemma \ref{lemma:winding} to show that none of the remaining manifolds can be homeomorphic to the complement of a link with winding number $\tilde{w}=0$. We can immediately rule out the candidates which are identified by SnapPy to be the complements of links with non-zero linking number. Moreover, those possessing a component with multiple distinct $\mathbb{S}^1 \times \mathbb{D}^2$ fillings can be eliminated by Lemma \ref{lemma:berge-gabai}: in fact, with respect to the homology basis used by SnapPy, the pair of fillings given by $(a_1,b_1)=(1,0)$ and $(\hat{a}_1,\hat{b}_1)=(0,1)$ works in each of these cases. This is all summarised in Table \ref{table:elimination}. 

        In general (just excluding $m412$ and $s596$), we can skip straight to the method of eliminating a manifold by Lemma \ref{lemma:free}. Each of the manifolds $M$ on the list satisfies $M((1,0),\ast)\cong\mathbb{S}^1\times\mathbb{D}^2$, which in turn has a unique filling slope $(a,b)$ that produces $M((1,0),(a,b))=\mathbb{S}^1\times\mathbb{S}^2$. It is then easy to check that $H_1(M((0,1), (a,b)))$ is finite. Table \ref{table:volumes} in the appendix contains $(a,b)$ for each $M$ on the list, which is ordered by volume and split into portions $L_k$ with volume between $V_{k-1}$ and $V_k$. 
        
        We are left with the Whitehead link complement, $m129$, as the only remaining option for $M \cong \mathbb{S}^3_{\tilde{P}}$. By Lemma \ref{lemma:twisted}, $\tilde{P}=W^\pm_{t'}$ for some $t'\in\mathbb{Z}$. 
    \end{proof}
\end{proposition} 

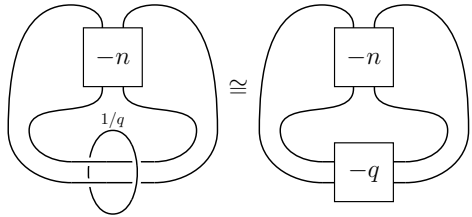
\begin{figure}[htbp!] 
    \centering 
    \resizebox{0.5\textwidth}{!}{\centering
\begin{tikzpicture} [squarednode/.style={rectangle, draw=black, fill=white, thick, minimum size=40pt}]
\begin{scope}[xshift=-3cm]
    \begin{knot}[end tolerance=1pt] 
    \flipcrossings{1,3} 
        \strand[white, double=black, thick, double distance=1pt] (-1,-1.25) 
        to [out=left, in=down] (-2.5,0) 
        to [out=up, in=left] (-1,3)
        to [out=right, in=up] (-0.25,2.5)
        to [out=down, in=up] (-0.25,1) 
        to [out=down, in=up] (-2,0)
        to [out=down, in=left] (-1,-0.75); 
        \strand[white, double=black, thick, double distance=1pt] (-1,-0.75) to [out=right, in=left] (1,-0.75); 
        \strand[white, double=black, thick, double distance=1pt] (1,-0.75) 
        to [out=right, in=down] (2,0)
        to [out=up, in=down] (0.25,1)
        to [out=up, in=down] (0.25,2.5) 
        to [out=up, in=left] (1,3) 
        to [out=right, in=up] (2.5,0) 
        to [out=down, in=right] (1,-1.25); 
        \strand[white, double=black, thick, double distance=1pt] (1,-1.25) to [out=left, in=right] (-1,-1.25); 
        \strand[white, double=black, thick, double distance=1pt] (0,0) 
        to [out=right, in=right] (0,-2) 
        to [out=left, in=left] (0,0); 
        \node[] at (0,0.25) {\large $1/q$}; 
    \end{knot} 
    \node[squarednode, fill=white, opacity=1, draw=black] at (0,1.75) {\Large $-n$}; 
\end{scope} 
\node[] at (0,1) {\Large $\cong$}; 
\begin{scope}[xshift=3cm]
    \begin{knot}[end tolerance=1pt] 
        \strand[white, double=black, thick, double distance=1pt] (-1,-1.25) 
        to [out=left, in=down] (-2.5,0) 
        to [out=up, in=left] (-1,3)
        to [out=right, in=up] (-0.25,2.5)
        to [out=down, in=up] (-0.25,1) 
        to [out=down, in=up] (-2,0)
        to [out=down, in=left] (-1,-0.75); 
        \strand[white, double=black, thick, double distance=1pt] (-1,-0.75) to [out=right, in=left] (1,-0.75); 
        \strand[white, double=black, thick, double distance=1pt] (1,-0.75) 
        to [out=right, in=down] (2,0)
        to [out=up, in=down] (0.25,1)
        to [out=up, in=down] (0.25,2.5) 
        to [out=up, in=left] (1,3) 
        to [out=right, in=up] (2.5,0) 
        to [out=down, in=right] (1,-1.25); 
        \strand[white, double=black, thick, double distance=1pt] (1,-1.25) to [out=left, in=right] (-1,-1.25); 
    \end{knot} 
    \node[squarednode, fill=white, opacity=1, draw=black] at (0,-1) {\Large $-q$}; 
    \node[squarednode, fill=white, opacity=1, draw=black] at (0,1.75) {\Large $-n$}; 
\end{scope}
\end{tikzpicture} } 
    \caption{Rolfsen $-q$-twist giving $\mathbb{S}^3_{W^n}(1/q) \cong \mathbb{S}^3_{T^n_q}$. }
    \label{figure:1/q} 
\end{figure} 

\begin{remark} 
    Note that we can also work in the opposite direction: starting from any volume bound $V_{\max}$, we can recover the corresponding lower bound on the denominator via the formula $$q_{\min} := \left\lceil 2\pi \sqrt{\frac{3}{1 - (V_{\max}^{-1} \cdot \vol(\mathbb{S}^3_{W^\pm_t}))^{2/3}}} \ \right\rceil. $$
\end{remark}

Finally, having successfully narrowed down our hyperbolic pattern to a twisted Whitehead link, \mbox{$\tilde{P}=W^\pm_{t'}$}, we proceed to complete the proof of Theorem \ref{theorem:refined}. 

\begin{proof}[Proof of Theorem \ref{theorem:refined}]
    By Proposition \ref{proposition:pattern}, we know that $K'=P'(J)$ is a satellite knot and that there is a slope-preserving homeomorphism $\mathbb{S}^3_{W^\pm_t}(p/q)\cong\mathbb{S}^3_{P'}(p/q)$, where $P'$ is either (i) a hyperbolic pattern itself or (ii) a cable $C_{r,s}(\hat{P})$ of a hyperbolic pattern $\hat{P}$. Let $\tilde{P}$ denote the hyperbolic pattern in each case and consider the slope-preserving homeomorphism \mbox{$\mathbb{S}^3_{W^\pm_t}(p/q)\cong\mathbb{S}^3_{\tilde{P}}(\tilde{p}/\tilde{q})$}. We have just shown in Proposition \ref{proposition:elimination} that $\tilde{P}=W^\pm_{t'}$ for some $t'\in\mathbb{Z}$. In case (ii), where $(\tilde{P},\tilde{p}/\tilde{q})=(\hat{P},p/qs^2)$, we have $P' = C_{\pm1,s}(W^\pm_t)$ by Lemma \ref{lemma:cabled_twisted}; however, we can rule this out by Corollary \ref{corollary:csc}. In case (i), where $(\tilde{P},\tilde{p}/\tilde{q})=(P',p/q)$, Lemma \ref{lemma:knotting} tells us that $V^\pm_t=V^\pm_{t'}$, so $t=t'$ and $K=K'$, as required.  
\end{proof}

\section{Non-characterising slopes} 
\label{section:non-characterising} 

In this final section, we construct the non-characterising slopes that are described in Theorem \ref{theorem:non-characterising} and depicted in Figure \ref{figure:$W^n(T^m_q)$}.  

\begin{theorem:non-characterising} 
    Let $K = W^n(T^m_q)$ and $K' = W^m(T^n_q)$ for any choices of $q,m,n\in\mathbb{Z}\setminus\{0\}$. 

    Then there exists an orientation-preserving homeomorphism $\mathbb{S}^3_K(1/q)\cong\mathbb{S}^3_{K'}(1/q)$. 

    In particular, when $m \neq n$, $1/q$ is a non-characterising slope for both $K$ and $K'$. 
\end{theorem:non-characterising} 

We employ Brakes' strategy of gluing together a pair of knot complements to construct pairs of knots which share a common surgery \cite{brakes}. This demonstrates precisely why we require the $|p|\neq1$ condition in Theorem \ref{theorem:refined}.

\subsection{Construction} 

Let $W^n = V^n \cup U$ be the $n$-clasped untwisted Whitehead link. We prohibit any extra twisting of $V^n$ inside the solid torus $\mathbb{S}^3_U$ because we require our pattern to be unknotted in order to carry out the upcoming construction. Recall that $W^n$ has winding number $w=0$ and that the two unknotted components can be exchanged by an isotopy. We clearly require $n\neq0$ in order for this pattern to be non-trivial. 

Let $T^m_q$ denote the $(m,q)$-double twist knot. We take $m\neq0$ and $q\neq0$ in order to ensure that $T^m_q$ is not the unknot. Note that $T^m_q = T^q_m$. Also, observe that by the Rolfsen $-q$-twist depicted in Figure \ref{figure:1/q}, $T^m_q$ can be obtained by performing $1/q$-surgery on one component of the $m$-clasped Whitehead link; it does not matter which component, since they are both unknotted and can be exchanged by an isotopy. 

Now, by applying an $n$-clasped Whitehead link pattern $W^n$ to an $(m,q)$-double twist knot $T^m_q$, we can construct an infinite family of pairs of satellite knots sharing a $1/q$-surgery. The idea is that, since filling the pattern piece $\mathbb{S}^3_{W^n}$ with slope $1/q$ produces an $(n,q)$-double twist knot complement, $\mathbb{S}^3_{W^n}(1/q)\cong\mathbb{S}^3_{T^n_q}$, it is possible for a homeomorphism between surgeries to switch the JSJ pieces. 

As a special case, we generate the following infinite family of pairs of true Whitehead doubles of true twist knots which each share a non-characterising slope of the form $1/q$. 

\begin{corollary} 
\label{corollary: non-characterising}
    Let $K_q=W^+(T^-_q)$ and $K'_q=W^-(T^+_q)$ for any choice of $q\in\mathbb{Z}\setminus\{0\}$. 

    Then there exists an orientation-preserving homeomorphism $\mathbb{S}^3_{K_q}(1/q)\cong\mathbb{S}^3_{K'_q}(1/q)$. 

    In particular, $1/q$ is a non-characterising slope for both $K_q$ and $K'_q$. 
\end{corollary} 

Furthermore, when $q=\pm1$, we obtain the integer non-characterising slopes $\pm1$ and rediscover an example given in \cite{brakes}. The following pairs consist of orientation-preserving homeomorphic fillings which both admit JSJ decompositions into a trefoil knot complement and a figure eight knot complement, glued together along the central JSJ torus by an orientation-reversing map which exchanges meridians and longitudes. Note that these examples are mirrors of one another. 

\begin{example}
    Let $K_{\pm}=W^+(T^-_{\pm})$ and $K'_{\pm}=W^-(T^+_{\pm})$. 
    
    Then there exists an orientation-preserving homeomorphism $\mathbb{S}^3_{K_{\pm}}(\pm1)\cong\mathbb{S}^3_{K'_{\pm}}(\pm1)$. 

    In particular, $\pm1$ is a non-characterising slope for both $K_{\pm}$ and $K'_{\pm}$. 
\end{example} 

Figure \ref{figure:examples} shows the knots $K_{\pm}$ and $K'_{\pm}$.

\subsection{Gluing maps} 

In order to compare JSJ decompositions, it is crucial to consider gluing maps as well as JSJ pieces.

Let $K=W^n(T^m_q)$. We will carefully analyse exactly what the JSJ decomposition of $\mathbb{S}^3_K(1/q)$ looks like. We already know that there are precisely two JSJ pieces, which each take the form of a double twist knot complement: the companion piece $\mathbb{S}^3_{T^m_q}$ and the filled pattern piece $\mathbb{S}^3_{W^n}(1/q)\cong\mathbb{S}^3_{T^n_q}$. Note that if we take $|q|\geq2$, then these are both hyperbolic. 

To describe the gluing map, we use the satellite construction to compare how the meridian and longitude of both of the double twist knots interact on the JSJ torus along which their complements meet. In particular, we show that the gluing map identifies these in pairs, and that each pair represents a nullhomologous curve on either side of the JSJ torus. 

Consider the $n$-clasped untwisted Whitehead link pattern, $W^n = V^n \cup U$. We can view the component $V^n \subset \mathbb{S}^3_U \cong \mathbb{S}^1 \times \mathbb{D}^2$ as sitting inside a solid torus with meridian $\mu$ and longitude $\lambda$. When we form the $n$-clasped Whitehead double of the companion $(m,q)$-double twist knot $T^m_q$, we make the following identifications: 
\begin{align*}
\mu = \lambda_U = \mu_{T^m_q}, \\
\lambda = \mu_U = \lambda_{T^m_q}. 
\end{align*}

The complement of $K$ can now be expressed as $\mathbb{S}^3_K \cong \mathbb{S}^3_{T^m_q} \cup_{\mathbb{T}^2} \mathbb{S}^3_{W^n}$, where the union along the central JSJ torus $\mathbb{T}^2$ is specified by these identifications. 

Now let's perform $1/q$-surgery on $K$. Filling the pattern piece gives $\mathbb{S}^3_{W^n}(1/q)\cong\mathbb{S}^3_{T^n_q}$: we can see this by performing an isotopy so that the component labelled $U$ is now sitting inside the solid torus $\mathbb{S}^3_{V^n} \cong \mathbb{S}^1 \times \mathbb{D}^2$, and then viewing $1/q$-surgery along $V^n$ as applying a Rolfsen $-q$-twist to $U$ inside $\mathbb{S}^3_{V^n}$, as shown in Figure \ref{figure:1/q}. This transforms $U$ into the $(n,q)$-double twist knot $T^n_q$. In particular, observe that we now have the following identifications: 
\begin{align*}
\mu_{T^m_q} = \lambda_U=\lambda_{T^n_q}, \\
\lambda_{T^m_q} = \mu_U=\mu_{T^n_q}. 
\end{align*}

In other words, our gluing map exchanges meridians and longitudes. Note that each of these is nullhomologous on one side of the JSJ torus: the curve $\mu_{T^m_q}=\lambda_{T^n_q}$ on the pattern side by the $w=0$ case of Lemma \ref{lemma:gordon}, and the curve $\lambda_{T^m_q}=\mu_{T^n_q}$ on the companion side since it bounds a Seifert surface.

\subsection{Proof of Theorem \ref{theorem:non-characterising}} 

\begin{proof}[Proof of Theorem \ref{theorem:non-characterising}]
    Let $K=W^n(T^m_q)$ and $K'=W^m(T^n_q)$. Clearly, both $\mathbb{S}^3_K(1/q)$ and $\mathbb{S}^3_{K'}(1/q)$ admit a JSJ decomposition into two double twist knot complements, $\mathbb{S}^3_{T^m_q}$ and $\mathbb{S}^3_{T^n_q}$. By the reasoning above, the gluing map in both cases is given by the identifications $\mu_{T^m_q}=\lambda_{T^n_q}$ and $\lambda_{T^m_q}=\mu_{T^n_q}$. Therefore $\mathbb{S}^3_K(1/q) \cong \mathbb{S}^3_{K'}(1/q)$, and $K\neq K'$ by the geometry of their complements. 
\end{proof}

\begin{proof}[Proof of Corollary \ref{corollary: non-characterising}]
    Simply set $n=+1$ and $m=-1$. The result holds for all $q\in\mathbb{Z}\setminus\{0\}$.
\end{proof} 

\begin{figure}[htbp!]
    \centering 
    \begin{subfigure}{\textwidth} 
        \centering 
        \resizebox{0.6\textwidth}{!}{\centering
\begin{tikzpicture} [squarednode/.style={rectangle, draw=black, fill=white, thick, minimum size=40pt}]
    \begin{knot}[end tolerance=1pt] 
        \strand[black, double=white, thick, double distance=4pt] (-1,-1) 
        to [out=left, in=down] (-2,0.25);  
        \strand[black, thick] (-2.085,0.25) 
        to [out=up, in=down] (-2.6,1)
        to [out=up, in=left] (-2.3,1.15); 
        \strand[black, thick] (-2.3,1.35) 
        to [out=left, in=down] (-2.6,1.5) 
        to [out=up, in=down] (-2.085,2.25); 
        \strand[black, thick] (-1.915,0.25) 
        to [out=up, in=down] (-1.4,1)
        to [out=up, in=right] (-1.7,1.15); 
        \strand[black, thick] (-1.7,1.35) 
        to [out=right, in=down] (-1.4,1.5) 
        to [out=up, in=down] (-1.915,2.25);         
        \strand[black, double=white, thick, double distance=4pt] (-2,2.25) 
        to [out=up, in=left] (-1,3)
        to [out=right, in=up] (-0.25,2.5)
        to [out=down, in=up] (-0.25,1) 
        to [out=down, in=up] (-1.5,0)
        to [out=down, in=left] (-1,-0.5); 
        \strand[black, double=white, thick, double distance=4pt] (-1,-0.5) 
        to [out=right, in=left] (1,-0.5); 
        \strand[black, double=white, thick, double distance=4pt] (1,-0.5) 
        to [out=right, in=down] (1.5,0)
        to [out=up, in=down] (0.25,1)
        to [out=up, in=down] (0.25,2.5) 
        to [out=up, in=left] (1,3) 
        to [out=right, in=up] (2,2) 
        to [out=down, in=up] (2,1.6)
        to [out=down, in=up] (2,0.9) 
        to [out=down, in=up] (2,0) 
        to [out=down, in=right] (1,-1); 
        \strand[black, double=white, thick, double distance=4pt] (1,-1) 
        to [out=left, in=right] (-1,-1); 
    \end{knot} 
    \node[squarednode, fill=white, opacity=1] at (0,-0.75) {\Large $-1$}; 
    \node[squarednode, fill=white, opacity=1] at (0,1.75) {\Large $+1$}; 
    \node[rectangle, draw=black, fill=white, thick, minimum size=20pt] at (-2,1.25) {$-1$}; 
    \node[rectangle, draw=black, fill=white, fill opacity=1, thick, minimum size=20pt] at (2,1.25) {$0$}; 
\end{tikzpicture} 
\qquad 
\begin{tikzpicture} [squarednode/.style={rectangle, draw=black, fill=white, thick, minimum size=40pt}]
    \begin{knot}[end tolerance=1pt] 
        \strand[black, double=white, thick, double distance=4pt] (-1,-1) 
        to [out=left, in=down] (-2,0.25);  
        \strand[black, thick] (-2.085,0.25) 
        to [out=up, in=down] (-2.6,1)
        to [out=up, in=left] (-2.3,1.15); 
        \strand[black, thick] (-2.3,1.35) 
        to [out=left, in=down] (-2.6,1.5) 
        to [out=up, in=down] (-2.085,2.25); 
        \strand[black, thick] (-1.915,0.25) 
        to [out=up, in=down] (-1.4,1)
        to [out=up, in=right] (-1.7,1.15); 
        \strand[black, thick] (-1.7,1.35) 
        to [out=right, in=down] (-1.4,1.5) 
        to [out=up, in=down] (-1.915,2.25);         
        \strand[black, double=white, thick, double distance=4pt] (-2,2.25) 
        to [out=up, in=left] (-1,3)
        to [out=right, in=up] (-0.25,2.5)
        to [out=down, in=up] (-0.25,1) 
        to [out=down, in=up] (-1.5,0)
        to [out=down, in=left] (-1,-0.5); 
        \strand[black, double=white, thick, double distance=4pt] (-1,-0.5) 
        to [out=right, in=left] (1,-0.5); 
        \strand[black, double=white, thick, double distance=4pt] (1,-0.5) 
        to [out=right, in=down] (1.5,0)
        to [out=up, in=down] (0.25,1)
        to [out=up, in=down] (0.25,2.5) 
        to [out=up, in=left] (1,3) 
        to [out=right, in=up] (2,2) 
        to [out=down, in=up] (2,1.6)
        to [out=down, in=up] (2,0.9) 
        to [out=down, in=up] (2,0) 
        to [out=down, in=right] (1,-1); 
        \strand[black, double=white, thick, double distance=4pt] (1,-1) 
        to [out=left, in=right] (-1,-1); 
    \end{knot} 
    \node[squarednode, fill=white, opacity=1] at (0,-0.75) {\Large $-1$}; 
    \node[squarednode, fill=white, opacity=1] at (0,1.75) {\Large $-1$}; 
    \node[rectangle, draw=black, fill=white, thick, minimum size=20pt] at (-2,1.25) {$+1$}; 
    \node[rectangle, draw=black, fill=white, fill opacity=1, thick, minimum size=20pt] at (2,1.25) {$-4$}; 
\end{tikzpicture} }
        \caption{A pair of knots, $K_+=W^+(S)$ and $K'_+=W^-(+T)$, with $\mathbb{S}^3_{K_+}(+1) \cong \mathbb{S}^3_{K'_+}(+1)$. }
        \label{subfigure:+example} 
    \end{subfigure}
    \begin{subfigure}{\textwidth} 
        \centering 
        \resizebox{0.6\textwidth}{!}{\centering
\begin{tikzpicture} [squarednode/.style={rectangle, draw=black, fill=white, thick, minimum size=40pt}]
    \begin{knot}[end tolerance=1pt] 
        \strand[black, double=white, thick, double distance=4pt] (-1,-1) 
        to [out=left, in=down] (-2,0.25);  
        \strand[black, thick] (-2.085,0.25) 
        to [out=up, in=down] (-2.6,1)
        to [out=up, in=left] (-2.3,1.15); 
        \strand[black, thick] (-2.3,1.35) 
        to [out=left, in=down] (-2.6,1.5) 
        to [out=up, in=down] (-2.085,2.25); 
        \strand[black, thick] (-1.915,0.25) 
        to [out=up, in=down] (-1.4,1)
        to [out=up, in=right] (-1.7,1.15); 
        \strand[black, thick] (-1.7,1.35) 
        to [out=right, in=down] (-1.4,1.5) 
        to [out=up, in=down] (-1.915,2.25);         
        \strand[black, double=white, thick, double distance=4pt] (-2,2.25) 
        to [out=up, in=left] (-1,3)
        to [out=right, in=up] (-0.25,2.5)
        to [out=down, in=up] (-0.25,1) 
        to [out=down, in=up] (-1.5,0)
        to [out=down, in=left] (-1,-0.5); 
        \strand[black, double=white, thick, double distance=4pt] (-1,-0.5) 
        to [out=right, in=left] (1,-0.5); 
        \strand[black, double=white, thick, double distance=4pt] (1,-0.5) 
        to [out=right, in=down] (1.5,0)
        to [out=up, in=down] (0.25,1)
        to [out=up, in=down] (0.25,2.5) 
        to [out=up, in=left] (1,3) 
        to [out=right, in=up] (2,2) 
        to [out=down, in=up] (2,1.6)
        to [out=down, in=up] (2,0.9) 
        to [out=down, in=up] (2,0) 
        to [out=down, in=right] (1,-1); 
        \strand[black, double=white, thick, double distance=4pt] (1,-1) 
        to [out=left, in=right] (-1,-1); 
    \end{knot} 
    \node[squarednode, fill=white, opacity=1] at (0,-0.75) {\Large $+1$}; 
    \node[squarednode, fill=white, opacity=1] at (0,1.75) {\Large $+1$}; 
    \node[rectangle, draw=black, fill=white, thick, minimum size=20pt] at (-2,1.25) {$-1$}; 
    \node[rectangle, draw=black, fill=white, fill opacity=1, thick, minimum size=20pt] at (2,1.25) {$+4$}; 
\end{tikzpicture} 
\qquad 
\begin{tikzpicture} [squarednode/.style={rectangle, draw=black, fill=white, thick, minimum size=40pt}]
    \begin{knot}[end tolerance=1pt] 
        \strand[black, double=white, thick, double distance=4pt] (-1,-1) 
        to [out=left, in=down] (-2,0.25);  
        \strand[black, thick] (-2.085,0.25) 
        to [out=up, in=down] (-2.6,1)
        to [out=up, in=left] (-2.3,1.15); 
        \strand[black, thick] (-2.3,1.35) 
        to [out=left, in=down] (-2.6,1.5) 
        to [out=up, in=down] (-2.085,2.25); 
        \strand[black, thick] (-1.915,0.25) 
        to [out=up, in=down] (-1.4,1)
        to [out=up, in=right] (-1.7,1.15); 
        \strand[black, thick] (-1.7,1.35) 
        to [out=right, in=down] (-1.4,1.5) 
        to [out=up, in=down] (-1.915,2.25);         
        \strand[black, double=white, thick, double distance=4pt] (-2,2.25) 
        to [out=up, in=left] (-1,3)
        to [out=right, in=up] (-0.25,2.5)
        to [out=down, in=up] (-0.25,1) 
        to [out=down, in=up] (-1.5,0)
        to [out=down, in=left] (-1,-0.5); 
        \strand[black, double=white, thick, double distance=4pt] (-1,-0.5) 
        to [out=right, in=left] (1,-0.5); 
        \strand[black, double=white, thick, double distance=4pt] (1,-0.5) 
        to [out=right, in=down] (1.5,0)
        to [out=up, in=down] (0.25,1)
        to [out=up, in=down] (0.25,2.5) 
        to [out=up, in=left] (1,3) 
        to [out=right, in=up] (2,2) 
        to [out=down, in=up] (2,1.6)
        to [out=down, in=up] (2,0.9) 
        to [out=down, in=up] (2,0) 
        to [out=down, in=right] (1,-1); 
        \strand[black, double=white, thick, double distance=4pt] (1,-1) 
        to [out=left, in=right] (-1,-1); 
    \end{knot} 
    \node[squarednode, fill=white, opacity=1] at (0,-0.75) {\Large $+1$}; 
    \node[squarednode, fill=white, opacity=1] at (0,1.75) {\Large $-1$}; 
    \node[rectangle, draw=black, fill=white, thick, minimum size=20pt] at (-2,1.25) {$+1$}; 
    \node[rectangle, draw=black, fill=white, fill opacity=1, thick, minimum size=20pt] at (2,1.25) {$0$}; 
\end{tikzpicture} }
        \caption{A pair of knots, $K_-=W^+(-T)$ and $K'_-=W^-(S)$, with $\mathbb{S}^3_{K_-}(-1) \cong \mathbb{S}^3_{K'_-}(-1)$. }
        \label{subfigure:-example}         
    \end{subfigure} 
    \caption{Pairs of knots, $K_\pm$ and $K'_\pm$, with $\mathbb{S}^3_{K_{\pm}}(\pm1) \cong \mathbb{S}^3_{K'_{\pm}}(\pm1)$.}
    \label{figure:examples} 
\end{figure}
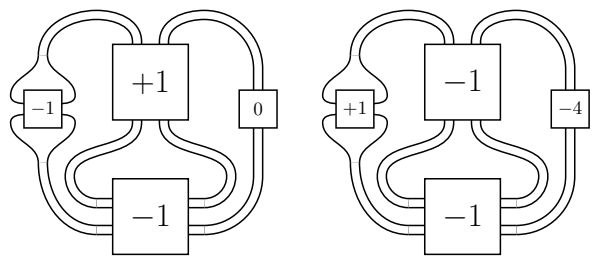
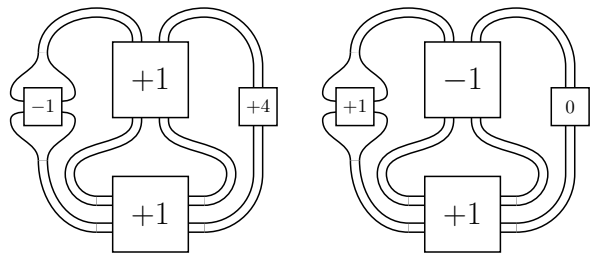

\printbibliography

\appendix

\newpage

\section{Minimal geodesics} 

In Table \ref{table:twist_lengths}, we give the values of $\mathfrak{q}(\mathbb{S}^3_K)$ obtained using the minimal geodesics method when $K=T^\pm_t$ is a $\pm1$-clasped $t$-twist knot for $|t|\leq34$. Since we have shown that $\Lambda^\pm_t<0.14$ for $|t|\geq35$, this completes the proof of Corollary \ref{corollary:new}. 

\begin{table}[htbp!] 
\centering 
    \begin{subtable}[htbp!]{0.49\textwidth}
    \centering 
        \begin{tabular}{|c|c|c|} \hline 
            $t$ & $\Lambda^\pm_t = \sys(\mathbb{S}^3_{T^\pm_t})$ & $\mathfrak{q}(\mathbb{S}^3_{T^\pm_t})$ \\ \hline \hline  
            $\mp1$ & $1.087070144996$ & $21$ \\ \hline 
            $\mp2$ & $0.330635521631$ & $27$ \\ \hline 
            $\mp3$ & $0.153578692788$ & $34$ \\ \hline \hline 
            $\mp4$ & $0.088638567733$ & $42$ \\ \hline 
            $\mp5$ & $0.057710538310$ & $51$ \\ \hline 
            $\mp6$ & $0.040571176947$ & $60$ \\ \hline 
            $\mp7$ & $0.030082564226$ & $68$ \\ \hline 
            $\mp8$ & $0.023196961993$ & $77$ \\ \hline 
            $\mp9$ & $0.018433283838$ & $86$ \\ \hline 
            $\mp10$ & $0.015000689183$ & $95$ \\ \hline 
            $\mp11$ & $0.012445426919$ & $104$ \\ \hline 
            $\mp12$ & $0.010491939188$ & $113$ \\ \hline 
            $\mp13$ & $0.008964992283$ & $122$ \\ \hline 
            $\mp14$ & $0.007748815892$ & $131$ \\ \hline 
            $\mp15$ & $0.006764425705$ & $140$ \\ \hline 
            $\mp16$ & $0.005956433725$ & $149$ \\ \hline 
            $\mp17$ & $0.005285065287$ & $158$ \\ \hline 
            $\mp18$ & $0.004721159018$ & $167$ \\ \hline 
            $\mp19$ & $0.004242940361$ & $176$ \\ \hline 
            $\mp20$ & $0.003833884329$ & $185$ \\ \hline 
            $\mp21$ & $0.003481266543$ & $194$ \\ \hline 
            $\mp22$ & $0.003175160372$ & $203$ \\ \hline 
            $\mp23$ & $0.002907729811$ & $212$ \\ \hline 
            $\mp24$ & $0.002672722444$ & $221$ \\ \hline 
            $\mp25$ & $0.002465100316$ & $230$ \\ \hline 
            $\mp26$ & $0.002280767426$ & $239$ \\ \hline 
            $\mp27$ & $0.002116366001$ & $248$ \\ \hline 
            $\mp28$ & $0.001969122376$ & $257$ \\ \hline 
            $\mp29$ & $0.001836729164$ & $266$ \\ \hline 
            $\mp30$ & $0.001717254246$ & $275$ \\ \hline 
            $\mp31$ & $0.001609069838$ & $284$ \\ \hline 
            $\mp32$ & $0.001510796722$ & $293$ \\ \hline 
            $\mp33$ & $0.001421260048$ & $303$ \\ \hline 
            $\mp34$ & $0.001339454037$ & $312$ \\ \hline             
        \end{tabular} 
    \end{subtable} 
    \begin{subtable}[htbp!]{0.49\textwidth}
    \centering 
        \begin{tabular}{|c|c|c|} \hline 
            $t$ & $\Lambda^\pm_t = \sys(\mathbb{S}^3_{T^\pm_t})$ & $\mathfrak{q}(\mathbb{S}^3_{T^\pm_t})$ \\ \hline \hline 
            $\pm1$ & $-$ & $-$ \\ \hline 
            $\pm2$ & $0.562399148646$ & $23$ \\ \hline 
            $\pm3$ & $0.217101464988$ & $30$ \\ \hline \hline 
            $\pm4$ & $0.114457182528$ & $38$ \\ \hline 
            $\pm5$ & $0.070693861571$ & $47$ \\ \hline 
            $\pm6$ & $0.048010614991$ & $55$ \\ \hline 
            $\pm7$ & $0.034739404196$ & $64$ \\ \hline 
            $\pm8$ & $0.026304508293$ & $73$ \\ \hline 
            $\pm9$ & $0.020609961033$ & $82$ \\ \hline 
            $\pm10$ & $0.016584439142$ & $90$ \\ \hline 
            $\pm11$ & $0.013633629114$ & $99$ \\ \hline 
            $\pm12$ & $0.011406166569$ & $108$ \\ \hline 
            $\pm13$ & $0.009683450721$ & $117$ \\ \hline 
            $\pm14$ & $0.008323669024$ & $126$ \\ \hline 
            $\pm15$ & $0.007231550219$ & $135$ \\ \hline 
            $\pm16$ & $0.006341162326$ & $144$ \\ \hline 
            $\pm17$ & $0.005605698492$ & $153$ \\ \hline 
            $\pm18$ & $0.004991184252$ & $162$ \\ \hline 
            $\pm19$ & $0.004472474961$ & $171$ \\ \hline 
            $\pm20$ & $0.004030637877$ & $180$ \\ \hline 
            $\pm21$ & $0.003651197189$ & $189$ \\ \hline 
            $\pm22$ & $0.003322931515$ & $198$ \\ \hline 
            $\pm23$ & $0.003037033650$ & $208$ \\ \hline 
            $\pm24$ & $0.002786512978$ & $217$ \\ \hline 
            $\pm25$ & $0.002565763633$ & $226$ \\ \hline 
            $\pm26$ & $0.002370247869$ & $235$ \\ \hline 
            $\pm27$ & $0.002196260793$ & $244$ \\ \hline 
            $\pm28$ & $0.002040753425$ & $253$ \\ \hline 
            $\pm29$ & $0.001901198105$ & $262$ \\ \hline 
            $\pm30$ & $0.001775485070$ & $271$ \\ \hline 
            $\pm31$ & $0.001661842188$ & $280$ \\ \hline 
            $\pm32$ & $0.001558772135$ & $289$ \\ \hline 
            $\pm33$ & $0.001465002792$ & $298$ \\ \hline 
            $\pm34$ & $0.001379447782$ & $307$ \\ \hline     
        \end{tabular}   
    \end{subtable} 
\caption{The length of a minimal geodesic in $\mathbb{S}^3_{T^\pm_t}$ and $\mathfrak{q}(\mathbb{S}^3_{T^\pm_t})$ for $|t|\leq34$. } 
\label{table:twist_lengths} 
\end{table} 

\newpage 

In Table \ref{table:double_lengths}, we list the values of $\mathfrak{q}(\mathbb{S}^3_{W^n})$ contributing to our bounds for all $n$-clasped $t$-twisted Whitehead doubles $K=W^n_t(J)$ with $|n|\leq34$ (note: this is independent of $t$ and $J$). Since we have shown that $\Lambda^{(n)}<0.14$ for $|n|\geq35$, this completes the proof of Corollary \ref{corollary:main}. 

\begin{table}[htbp!] 
\centering 
    \begin{tabular}{|c|c|c|} \hline 
    \centering 
        $|n|$ & $\Lambda^{(n)} = \sys(\mathbb{S}^3_{W^n})$ & $\mathfrak{q}(\mathbb{S}^3_{W^n_t})$ \\ \hline \hline 
            $1$ & $1.061275061905$ & $21$ \\ \hline 
            $2$ & $0.773114038508$ & $22$ \\ \hline 
            $3$ & $0.351756919644$ & $26$ \\ \hline 
            $4$ & $0.197729369929$ & $31$ \\ \hline \hline 
            $5$ & $0.126321972231$ & $37$ \\ \hline 
            $6$ & $0.087607243698$ & $43$ \\ \hline 
            $7$ & $0.064305470069$ & $49$ \\ \hline 
            $8$ & $0.049202303775$ & $55$ \\ \hline 
            $9$ & $0.038858005544$ & $61$ \\ \hline 
            $10$ & $0.031464306931$ & $67$ \\ \hline 
            $11$ & $0.025996893021$ & $73$ \\ \hline 
            $12$ & $0.021840285802$ & $79$ \\ \hline 
            $13$ & $0.018606572877$ & $86$ \\ \hline 
            $14$ & $0.016041419853$ & $92$ \\ \hline 
            $15$ & $0.013972441787$ & $98$ \\ \hline 
            $16$ & $0.012279441187$ & $105$ \\ \hline 
            $17$ & $0.010876534341$ & $111$ \\ \hline 
            $18$ & $0.009701032735$ & $117$ \\ \hline 
            $19$ & $0.008706311759$ & $124$ \\ \hline 
            $20$ & $0.007857112359$ & $130$ \\ \hline 
            $21$ & $0.007126371214$ & $136$ \\ \hline 
            $22$ & $0.006493036413$ & $143$ \\ \hline 
            $23$ & $0.005940533385$ & $149$ \\ \hline 
            $24$ & $0.005455668899$ & $155$ \\ \hline 
            $25$ & $0.005027835778$ & $162$ \\ \hline 
            $26$ & $0.004648427592$ & $168$ \\ \hline 
            $27$ & $0.004310402283$ & $175$ \\ \hline 
            $28$ & $0.004007952923$ & $181$ \\ \hline 
            $29$ & $0.003736256576$ & $187$ \\ \hline 
            $30$ & $0.003491280776$ & $194$ \\ \hline 
            $31$ & $0.003269633001$ & $200$ \\ \hline 
            $32$ & $0.003068442554$ & $206$ \\ \hline 
            $33$ & $0.002885267124$ & $213$ \\ \hline 
            $34$ & $0.002718018298$ & $219$ \\ \hline 
    \end{tabular} 
\caption{The length of a minimal geodesic in $\mathbb{S}^3_{W^n}$ and $\mathfrak{q}(\mathbb{S}^3_{W^n})$ for $|n|\leq34$. } 
\label{table:double_lengths} 
\end{table}

\newpage 

\section{Minimal volumes}

Table \ref{table:volumes} contains the smallest 2-cusped orientable hyperbolic 3-manifolds from the SnapPy census, grouped by triangulation complexity and then ordered by increasing volume. For each $M$, we also include $(a,b)$ for which $H_1(M((0,1),(a,b)))$ is finite. 

\begin{table*}[htbp!] 
\centering 
\begin{multicols}{2} 
\centering 
    \begin{tabular}[t]{|c|r|l|c|}  
        \hline 
        & $M$ & $\vol(M)$ & $(a,b)$ \\ 
        \hline \hline 
        $L_1$ & $ m129 $ & $ 3.6638623767 $ & \\ 
        & $ m125 $ & $ 3.6638623767 $ & $ (4, 3) $ \\ \hline \hline 
        $L_2$ & $ m202 $ & $ 4.0597664256 $ & $ (5, 3) $ \\
        & $ m203 $ & $ 4.0597664256 $ & $ (0, 1) $ \\ \hline \hline 
        $L_3$ & $ m292 $ & $ 4.4153324775 $ & $ (-4, 5) $ \\
        & $ m295 $ & $ 4.4153324775 $ & $ (3, 2) $ \\ \hline \hline 
        $L_4$ & $ m328 $ & $ 4.5559188899 $ & $ (5, 4) $ \\
        & $ m329 $ & $ 4.5559188899 $ & $ (7, 5) $ \\ \hline \hline 
        $L_5$ & $ m357 $ & $ 4.7254015851 $ & $ (11, 7) $ \\
        & $ m359 $ & $ 4.7254015851 $ & $ (1, 2) $ \\ 
        & $ m366 $ & $ 4.7494999819 $ & $ (-5, 7) $ \\
        & $ m367 $ & $ 4.7494999819 $ & $ (3, 1) $ \\ 
        & $ s441 $ & $ 4.7517019655 $ & $ (5, 6) $ \\
        & $ s443 $ & $ 4.7517019655 $ & $ (-4, 3) $ \\ \hline \hline 
        $L_6$ & $ m388 $ & $ 4.8511707573 $ & $ (5, 8) $ \\
        & $ m391 $ & $ 4.8511707573 $ & $ (1, 1) $ \\ \hline \hline 
        $L_7$ & $ s503 $ & $ 4.8937641326 $ & $ (-10, 7) $ \\
        & $ s506 $ & $ 4.8937641326 $ & $ (-6, 5) $ \\ 
        & $ v1060 $ & $ 4.9327140585 $ & $ (5, 4) $ \\
        & $ v1061 $ & $ 4.9327140585 $ & $ (6, 7) $ \\ 
        & $ s548 $ & $ 4.9767702943 $ & $ (-14, 9) $ \\
        & $ s549 $ & $ 4.9767702943 $ & $ (2, 3) $ \\ \hline 
    \end{tabular} 
\newline 
\vfill 
\columnbreak 
    \begin{tabular}[t]{|c|r|l|c|} 
        \hline 
        & $M$ & $\vol(M)$ & $(a,b)$ \\ 
        \hline \hline 
        $L_8$ & $ s568 $ & $ 5.0294944813 $ & $ (5, 3) $ \\
        & $ s569 $ & $ 5.0294944813 $ & $ (-7, 9) $ \\ 
        & $ t02501 $ & $ 5.0411812564 $ & $ (-6, 5) $ \\
        & $ t02502 $ & $ 5.0411812564 $ & $ (7, 8) $ \\ 
        & $ s576 $ & $ 5.0425492156 $ & $ (-11, 8) $ \\
        & $ s577 $ & $ 5.0425492156 $ & $ (9, 7) $ \\ 
        & $ s578 $ & $ 5.0448991629 $ & $ (5, 1) $ \\
        & $ s579 $ & $ 5.0448991629 $ & $ (-7, 10) $ \\ 
        & $ v1178 $ & $ 5.0533214945 $ & $ (-7, 6) $ \\
        & $ v1180 $ & $ 5.0533214945 $ & $ (13, 9) $ \\ 
        & $ m412 $ & $5.0747080321$ & --- \\ 
        & $ s596 $ & $ 5.0747080321 $ & --- \\ 
        & $ s601 $ & $ 5.0826538415 $ & $ (-7, 11) $ \\
        & $ s602 $ & $ 5.0826538415 $ & $ (-2, 1) $ \\ 
        & $ v1203 $ & $ 5.0990348432 $ & $ (-17, 11) $ \\
        & $ v1204 $ & $ 5.0990348432 $ & $ (3, 4) $ \\ 
        & $ s621 $ & $ 5.1062718035 $ & $ (-5, 2) $ \\
        & $ s622 $ & $ 5.1062718035 $ & $ (8, 11) $ \\ 
        & $ o9\_05655 $ & $ 5.1111665875 $ & $ (8, 9) $ \\
        & $ o9\_05656 $ & $ 5.1111665875 $ & $ (7, 6) $ \\ 
        & $ s637 $ & $ 5.1273136230 $ & $ (19, 12) $ \\
        & $ s638 $ & $ 5.1273136230 $ & $ (-1, 3) $ \\ 
        & $ s647 $ & $ 5.1379412019 $ & $ (-5, 3) $ \\ 
        & $ t02727 $ & $ 5.1401504513 $ & $ (-8, 7) $ \\ 
        & $ t02728 $ & $ 5.1401504513 $ & $ (16, 11) $ \\ 
        & $ v1252 $ & $ 5.1503497145 $ & $ (9, 11) $ \\
        & $ v1253 $ & $ 5.1503497145 $ & $ (-7, 5) $ \\ 
        & $ s660 $ & $ 5.1549263093 $ & $ (-8, 13) $ \\
        & $ s661 $ & $ 5.1549263093 $ & $ (-1, 2) $ \\ 
        & $ v1263 $ & $ 5.1621342201 $ & $ (-6, 1) $ \\
        & $ v1264 $ & $ 5.1621342201 $ & $ (-9, 13) $ \\ 
        & $ t02749 $ & $ 5.1676956678 $ & $ (4, 5) $ \\
        & $ t02750 $ & $ 5.1676956678 $ & $ (-20, 13) $ \\ \hline \hline 
        $L_9$ & $ v1284 $ & $ 5.1799776154 $ & 
        \\
        & $ v1285 $ & $ 5.1799776154 $ & 
        \\ \hline 
    \end{tabular} 
\end{multicols} 
\caption{2-cusped orientable hyperbolic 3-manifolds with low volume. } 
\label{table:volumes} 
\end{table*}

\vfill

\end{document}